\documentclass[a4paper]{article}

\usepackage[english]{babel}
\usepackage[utf8x]{inputenc}
\usepackage[T1]{fontenc}
\usepackage{lipsum}
\usepackage{blindtext}
\usepackage{graphicx,amssymb,amsmath,textcomp}
\usepackage{circuitikz}
\usepackage[a4paper,top=3cm,bottom=2cm,left=3.2cm,right=3.2cm,marginparwidth=1.75cm]{geometry}
\usepackage{tikz,tikz-cd}
\usetikzlibrary{matrix,arrows,positioning,calc}
\usepackage{verbatim}
\usepackage{palatino, mathpazo}
\usepackage[nottoc]{tocbibind}

\usepackage{amsmath,amsthm}
\usepackage{faktor}
\usepackage{xfrac} 
\usepackage{graphicx}
\usepackage[colorlinks=true, allcolors=purple]{hyperref}
\usepackage{mathtools}

\newcommand{\spec}{{\rm Spec}}

\newcommand{\N}{{\mathbb N}}
\newcommand{\Z}{{\mathbb Z}}

\newcommand{\K}{{\mathbb K}}
\newcommand{\tx}{\tilde{x}}
\newcommand{\ty}{\tilde{y}}

\newcommand{\dR}{{d_{dR}}}

\newcommand{\bfem}[1]{\textbf{\emph{#1}}}

\theoremstyle{plain}
\newtheorem{theorem}{Theorem}[section]

\title{Shifted Symplectic Fibrations and \\ Derived Thurston Theorem}

\author{Mehmet Fırat Arıkan\footnote{farikan@metu.edu.tr; \ Department of Mathematics, METU, 06800, Ankara, T\"urkiye}, \,Kadri İlker Berktav\footnote{berktav@metu.edu.tr; \ Department of Mathematics, METU, 06800, Ankara, T\"urkiye}, \,Efe İzbudak\footnote{efe.izbudak@studium.uni-hamburg.de; Dept. of
Mathematics, Universität Hamburg, 20146, Hamburg, Germany \\
\bfem{2020 Mathematics Subject Classification}: 14A20, 14A30, 14F08. \\ \bfem{Keywords}: Derived symplectic geometry, symplectic fibrations, Thurston's theorem}} 

\date{\vspace{-5ex}}

\begin{document}
    
\theoremstyle{plain}
\newtheorem{lemma}[theorem]{Lemma} 
\newtheorem{proposition}[theorem]{Proposition} 
\newtheorem{corollary}[theorem]{Corollary}
\newtheorem*{theorem*}{Theorem}
\newtheorem{thmx}{Theorem}
\renewcommand{\thethmx}{\Alph{thmx}}
\newtheorem{corx}[thmx]{Corollary}
\theoremstyle{definition}
\newtheorem{notations}[theorem]{Notations}
\newtheorem{remark}[theorem]{Remark}
\newtheorem{observation}[theorem]{Observation}
\newtheorem{definition}[theorem]{Definition}
\newtheorem{condition}[theorem]{Condition}
\newtheorem{example}[theorem]{Example}
\newtheorem{claim}[theorem]{Claim}
\newtheorem{construction}[theorem]{Construction}
\newtheorem{notation}[theorem]{Notation}

\let\pf\proof
\let\epf\endproof
\numberwithin{equation}{section}

\maketitle

\begin{abstract}
In classical symplectic geometry, under mild conditions, Thurston proved that one can construct a compatible symplectic form on the total space of a symplectic fibration with a connected symplectic base. Here we prove a derived symplectic analogue of this result, called the \textit{derived Thurston theorem}. More precisely, we show that if a morphism  $\pi: X \rightarrow S$ of derived stacks has a shifted symplectic fibration structure
and the target stack $S$ admits a shifted symplectic structure, then, under certain conditions, one can construct a shifted symplectic structure on a dense Zariski-open derived substack of the source stack $X$, compatible with $\pi$ in a sense similar to the classical case. In this derived setting, we also present numerous examples of shifted symplectic fibrations,   develop an affine model construction for such fibrations, and provide applications of the derived Thurston theorem.
 \end{abstract}
 
{\small \tableofcontents}


\section{Introduction and summary}

Symplectic fibrations are one of the central and well-studied objects in classical symplectic topology
, providing a framework in which the fibers carry symplectic structures preserved by the transition maps. Their importance lies in enabling both local and global analysis of phase spaces, particularly in Hamiltonian mechanics and integrable systems. As emphasized in \cite{Mcduff}, symplectic fibrations allow one to relate topology and dynamics through tools such as connections and reduction. 

To formalize the notion of symplectic fibrations, classical symplectic topology employs the following setting: Let $E$ be a locally trivial fiber bundle whose fibers $F$ and base $B$ are even-dimensional. Denote the projection map by $\pi:E\to B$, and let 
$G\subset \mathsf{Diff}(F)$ be the structure group. We simply write $(F\hookrightarrow E \xrightarrow{\pi} B)_G$ for such a fiber bundle data. 
If the fiber $F$ admits a \emph{symplectic form} $\omega$ (a closed non-degenerate $2$-form on $F$) and all transition maps are symplectomorphisms of $(F,\omega)$, that is, $G\subset \mathsf{Symp}(F,\omega)\subset \mathsf{Diff}(F)$, then the above fiber bundle is called a \bfem{symplectic fibration}. 

For a symplectic fibration, each fiber $F_b=\pi^{-1}(b)$ carries a naturally induced symplectic form $\omega_b$, obtained by pulling back $\omega$ via a local trivialization around  $b\in B$. This construction is well-defined and independent of the choice of trivialization, as the transition maps lie within $\mathsf{Symp}(F,\omega)$.
For such fibrations, Thurston gave a construction of a symplectic form on the total space. Thurston's theorem indeed formalizes a cohomological condition under which the total space of a compact symplectic fibration over a connected symplectic base admits a compatible symplectic form. Let us briefly recall the setup and some terminology before outlining Thurston's method. 

Given a symplectic fibration as above, suppose further that the total space $E$ admits a symplectic form $\Omega$. Such $\Omega$ is said to be \bfem{compatible with} the fibration  $\pi:E\to B$ if its restriction to each fiber $F_b$ is $\omega_b$ \cite[Definition 6.1.2]{Mcduff}. This compatibility condition implies that, for every $b\in B$, the symplectic manifold $(F_b,\omega_b)$ admits a symplectic embedding into $(E,\Omega)$. This notion of compatibility is motivated by the result asserting that, for a locally trivial fiber bundle $(F\hookrightarrow E \xrightarrow{\pi} B)_G$ with $E$ closed and $B$ connected, if $E$ admits a symplectic form $\Omega$ whose restriction to each fiber is symplectic, then $\pi:E\to B$ admits the structure of a symplectic fibration compatible with $\Omega$, see \cite[Lemma 6.1.3]{Mcduff}. 
\medskip 

An immediate question arises regarding the existence of the converse statement. Thanks to the aforementioned theorem of Thurston, if $\pi:E\to B$ is a symplectic fibration over a connected symplectic base, $E$ is compact and there is a class $a \in H^2(E)$ whose restriction to each fiber $F_b$ is $[\omega_b]$, then one can construct a symplectic form $\Omega$ on $E$ that is compatible with $\pi$ \cite[Theorem 6.1.4]{Mcduff}. For further details on symplectic topology and symplectic fibrations, we refer the reader to \cite{Mcduff}.

In the derived regime, on the other hand, one can also seek analogous definitions and results within the framework of \textit{derived symplectic geometry (DSG)} introduced and studied for more than a decade by   Pantev, Toën, Vaquié, and Vezzosi \cite{PTVV}.   Notably, numerous classical results in smooth symplectic geometry have also been promoted to  PTVV's DSG. For instance, a local Darboux model has been verified for DSG in \cite{Brav}, and a Lagrangian neighborhood theorem for shifted symplectic derived schemes has been given in \cite{JS}. Also, it has been known that shifted cotangent stacks are shifted symplectic \cite{Calaque2019}. In addition to those, there are many interesting results in PTVV's setting that have no classical counterparts, such as derived symplectic structures on fiber products of derived stacks and on mapping stacks (including the \textit{Lagrangian intersection theorem} and the \textit{AKSZ construction}). 

\paragraph{Results of the paper.} 
In the same spirit, this paper aims to investigate symplectic fibrations and relevant literature within the context of DSG. That task includes formally describing \textit{derived symplectic fibrations} and developing the\textit{ derived Thurston theorem}, along with several of its applications.

In order to define the terms and verify the related results (analogues to those in the smooth theory) in the most general DSG setting, we utilize the tools of relative de Rham theory 
and introduce \textit{$n$-shifted symplectic fibration structures} (Definition \ref{defn_shifted symp fibration}) generalizing the notion of symplectic fibrations in the smooth setting. 

In our setting,  $\mathbb{K} $  denotes an algebraically closed field of characteristic zero.  Suppose also that all derived $ \K $-schemes/stacks $ {X} $ are  {locally finitely presented.}  Then we first establish the following result as the derived analogue of the \textit{induced compatible symplectic fibration lemma} \cite[Lemma 6.1.3]{Mcduff}. 

\begin{thmx}[Induced compatible symplectic fibrations] \label{thm:2}
Let $\pi: X \rightarrow S$ be a morphism of derived Artin $\K$-stacks. If $X$ admits an $n$-shifted symplectic structure $\omega_X$ whose restriction to each geometric fiber $X_s$ defines an $n$-shifted symplectic structure on $X_s$, then $\pi$ admits an $n$-shifted symplectic fibration structure compatible with $\omega_X$ in the sense of Definition \ref{def:compatibility}.     
\end{thmx}

Inspired by the classical theory outlined above, one could ask about the \textit{converse problem}, namely whether we can construct a compatible (absolute) shifted symplectic structure on the source $X$ induced by the given relative shifted symplectic structure on a morphism $\pi:X\rightarrow S$ of derived stacks with a shifted symplectic target $S$. 

The answer will be affirmative, leading to the derived version of Thurston's theorem. More precisely, we prove:

\begin{thmx}[Derived Thurston Theorem] \label{thm:1}
    Let $n\in \Z$. Suppose that $X,S$ are two derived Artin $\K$-stacks locally of finite presentation and that $\pi: X \rightarrow (S,\omega')$ is a morphism of derived stacks, with $n$-shifted symplectic target $(S,\omega')$, such that $\pi$ carries an $n$-shifted symplectic fibration structure  $\omega$ in the sense of  Definition \ref{defn_shifted symp fibration}.
    
    Then one can construct, under certain conditions, an $n$-shifted symplectic structure $\omega_X$ on a dense Zariski-open derived substack of $X$, and on all of $X$ when $X$ is proper over $\K$, compatible with $\omega$ in the sense of Definition \ref{def:compatibility} (cf. Theorem \ref{thm: proof_thurston in dsg}).
\end{thmx}

As applications of Theorem \ref{thm:1}, we present in this paper several constructions of absolute shifted symplectic structures on certain symplectic fibrations, including conormal stacks, moment map quotients, and mapping stacks, each on a dense Zariski-open derived substack. More precisely, we prove:
\begin{corx}\label{cor: app to Thm1}

\begin{enumerate}
        \item If $f: X \rightarrow S$ is a morphism of derived Artin $\K$-stacks locally of finite presentation such that $S$ admits an $(n-1)$-shifted symplectic structure, then $\pi: N^*[n] f\longrightarrow S$ carries an $(n-1)$-shifted symplectic fibration structure. 
        
        Assume further that $N^*[n]f$ and $S$ are quasi-compact and that there exists an $(n-1)$-shifted closed 2-form $\Omega \in \mathcal{A}^{2,cl}(N^*[n] f, n-1)$ compatible with this fibration structure (Definition \ref{def:compatibility}), that is, whose restriction to each fiber $\pi^{-1}(s)$ is homotopic to the canonical $(n-1)$-shifted symplectic structure $\omega_s$ of the fiber. Then there exists an induced $(n-1)$-shifted symplectic structure on a dense Zariski-open derived substack of the $n$-shifted conormal $\K$-stack $N^*[n] f$ which is compatible with the fibration $\pi$ (cf. Example \ref{example:conormal stack}, Proposition \ref{prop: conormal}). 
        
        The result can also be extended verbatim to the morphisms $N_{\beta}^*[n] f \to T_{\beta}^*[n]S$ of twisted $\K$-stacks, for closed $1$-forms $\beta$ of degree $n+1$ on $S$ with a null-homotopy of $f^*\beta$ (cf. Corollary \ref{cor: conormal}).
        \item Let $(X, \omega_X)$ be a {$2$-shifted} symplectic derived Artin $\K$-stack equipped with a smooth action of $G$, and let $\mu: X \rightarrow {\mathfrak{g}^*[2]}$ be a \emph{{$2$-shifted} moment map}. Then $X/G \longrightarrow (BG, \omega_{BG})$ carries a $2$-shifted symplectic fibration structure such that, under certain conditions, there exists an induced {$2$-shifted}  symplectic structure on a dense Zariski-open derived substack of $X/G$ (cf. Example \ref{ex:moment map}, Proposition \ref{prop: moment quotient}).
       
        \item  Let $Y$ be an $n$-shifted symplectic derived Artin $\K$-stack and $X$ be a smooth projective Calabi-Yau $m$-fold. Suppose $f: L \rightarrow Y$ is an $n$-shifted Lagrangian morphism and $p: Y \rightarrow S$ is an $n$-shifted Lagrangian fibration such that the base $S$ admits an $(n-1)$-shifted symplectic structure. 
        
        Under certain conditions, there exists an $(n-m-1)$-shifted symplectic structure on a dense Zariski-open derived substack of the mapping stack $\mathsf{Map}(X,L)$ which is compatible with the symplectic fibration structure on $\Pi: \mathsf{Map}(X,L) \longrightarrow \mathsf{Map}(X,S)$ transgressed from that of $p \circ f: L \rightarrow S$ (cf. Proposition \ref{prop: mapping_stack_thurston}).
    \end{enumerate}
\end{corx}
Lastly, we also provide a canonical construction of a shifted symplectic structure on the source for an $n$-shifted symplectic fibration morphism of affine derived schemes. 
In brief, we have:

\begin{thmx} \label{thm:affine_model}
   There is a canonical affine model for an $n$-shifted symplectic fibration structure on the map $\spec\beta: \spec A \to (\spec B, \omega_B)$ of derived affine schemes, induced from a submersion $\beta:B\rightarrow A$ of standard form cdgas, such that $\spec A$ inherits a canonical shifted symplectic
  structure (cf. \S\,\ref{sec:affine symplectic fibration}). 
\end{thmx}

\noindent \textit{This paper is organized as follows:} We recall the related terminology from DSG in \S\,\ref{sec:terminology}. Then, in \S\,\ref{sec:Relative structure}, we discuss relative structures and present numerous examples of shifted symplectic fibrations based on work in the literature. We prove Theorem \ref{thm:2} in \S\,\ref{sec:Ind_symp_fib}, where we also introduce compatibility. The precise version of Theorem \ref{thm:1} (cf. Theorem \ref{thm: proof_thurston in dsg}) will be proven in \S\,\ref{sec:results}.  We study, in \S\,\ref{sec:applications}, applications of Theorem \ref{thm:1} and establish the proof of Corollary \ref{cor: app to Thm1}. We close the paper by providing, in \S\,\ref{sec:affine symplectic fibration}, an affine model of our construction given in the proof of Theorem \ref{thm: proof_thurston in dsg}, which establishes Theorem \ref{thm:affine_model}.

\paragraph{Conventions.} Throughout the paper, all cdgas will be graded in non-positive degrees and  over $ \mathbb{K}.$ We will always consider $ \K $-schemes/stacks and assume that all classical $ \K $-schemes are \emph{locally of finite type}, and that all derived $ \K $-schemes/stacks $ {X} $ are  \emph{locally finitely presented.}

\paragraph{Acknowledgments.}
This work was supported by the Scientific and Technological Research Council of Türkiye (TÜBİTAK) under the 1001-Scientific and Technological Research Projects Funding Program (Project No: 125F117). 

All authors warmly thank the \textit{Higher Structures Research Group} in the Math Department at
Middle East Technical University for creating a stimulating research environment.


\section{Recollection: Some derived symplectic geometry}  \label{sec:terminology}

This section gathers some background material on derived algebraic/symplectic geometry, including key results from the literature. We mostly follow \cite{Calaque2019, CPTVV, PTVV, ToenHAG}. 

\subsection{Derived schemes and Artin stacks}

Let us begin with our algebraic model. Denote by $ cdga_{\K}^{\leq 0}$  the $\infty$-category\footnote{We actually mean the \emph{$\infty$-category} associated to the model category $ cdga_{\K}^{\leq 0} $,  with its natural model structure    for which equivalences are quasi-isomorphisms, and fibrations are epimorphisms in strictly negative degrees.} of \bfem{commutative dg $\mathbb{K}$-algebras in non-positive degrees}, where an \textit{object} $A$ in $ cdga_{\K}^{\leq 0}$ consists of \begin{enumerate}
            \item [(1)] a collection $ \{A^i\}_{i\leq0}$, where each $A^i$ is the $\K$-vector space of degree $i$ elements, 
            \item [(2)] a $\K$-bilinear, associative, supercommutative multiplication $A^m \otimes A^n \xrightarrow{\cdot} A^{m+n}$, and
            \item  [(3)] a unique square-zero derivation of degree 1  (the differential) $d$  on $A$ satisfying the graded Leibniz rule
        \begin{equation*}
             d (a \cdot b)= (d a) \cdot b + (-1)^{|a|}a \cdot (d b) 
        \end{equation*} for all $a\in A^{|a|}, b\in A^{|b|}.$
        \end{enumerate} We denote such objects by $(A, d)$ or just $A$. Note that $A$ has a decomposition $A=\bigoplus_i A^i.$ 
Letting $ CAlg_{\mathbb{K}} $ be the \textit{category of commutative $\K$-algebras},
we then have the following diagram of higher spaces \cite{Vezz2011}:
\[\begin{tikzcd}
    {CAlg_{\mathbb{K}}} && Sets \\
    && Grpd \\
    {cdga_{\mathbb{K}}^{\leq 0}} && {Grpd_{\infty}}
    \arrow["schemes", from=1-1, to=1-3]
    \arrow["stacks"{description}, from=1-1, to=2-3]
    \arrow[from=1-1, to=3-1]
    \arrow["{higher \ stacks}"{description}, from=1-1, to=3-3]
    \arrow[from=1-3, to=2-3]
    \arrow[from=2-3, to=3-3]
    \arrow["{derived \ stacks}"', from=3-1, to=3-3]
\end{tikzcd}\]

This diagrammatic description reveals the general principle when we consider higher spaces; in this context, we usually refer to them via the sheaf-theoretic formulation and regard them as certain homotopy sheaves, as in the diagram above, leading to the following definitions.

\begin{definition}
 An object $X$ of $\mathsf{Fun}(CAlg_{\mathbb{K}}, Sets)$ is called an \bfem{affine $\K$-scheme} if $X\simeq \spec A$ for  $A\in CAlg_{\mathbb{K}}$; and a \bfem{$\K$-scheme} if it has a cover by Zariski open affine $\K$-schemes. Denote by $\mathsf{Aff}_{\K}$ the space of such affine $\mathbb{K}$-schemes. 

 Here,  $\spec A$ is the usual prime spectrum of $A\in CAlg_{\mathbb{K}}$ such that the equivalence of categories $\mathsf{Aff}^{\mathrm{op}}\simeq CAlg_{\mathbb{K}}$ is given by the  spectrum functor $\spec: A \mapsto \spec A \in \mathsf{Aff}$.
\end{definition}

In derived geometry, there also exists an appropriate concept of a \emph{spectrum functor} described as the right adjoint to the global algebra of functions functor 
\begin{equation*}
\Gamma: dSt_{\K} \leftrightarrows (cdga_{\K}^{\leq 0})^{\mathrm{op}}: \spec.
\end{equation*} 

Here $dSt_{\K}$ denotes the \bfem{$\infty$-category of derived stacks}, with objects being \emph{$\infty$-sheaves on the site $(dAff)^{\mathrm{op}}:= cdga_{\mathbb{K}}^{\leq 0}$}. We formally denote elements in this opposite category $ (cdga_{\mathbb{K}}^{\leq 0})^{\mathrm{op}} $ by $\spec A$. In the spirit of previous discussion, any derived stack $X$ can also be viewed as $\infty$-groupoid-valued homotopy sheaf, and hence we have:

\begin{definition}
    A \bfem{derived stack} $ {X} $ is  a  functor  $$ {X}:  cdga_{\K}^{\leq 0} \rightarrow  Grp_{\infty}, \ \ A\mapsto X(A) \simeq Map_{dStk_{\mathbb{K}}}(\spec A, {X}),$$ satisfying a descent condition.
For more details, we refer to \cite{ToenHAG}. 
\end{definition}

Let us now focus on particular, more tractable derived stacks with good properties.

\begin{definition}
    An object ${X}$ in $dSt_{\mathbb{K}}$ is called an \bfem{affine derived $\mathbb{K}$-scheme} if  $X\simeq \spec A $ for some cdga $A \in cdga_{\K}^{\leq 0}$. An object $ X$ in $dSt_{\mathbb{K}}$ is then called a \bfem{derived $\mathbb{K}$-scheme} if it can be covered by Zariski open affine derived $\mathbb{K}$-schemes.
\end{definition} 
It should then be noted that any affine derived $\mathbb{K}$-scheme $X$ can be corepresented by a cdga $A$ such that $\spec A$ is viewed as the functor \[ \spec A: B  \longmapsto Hom_{cdga_{\mathbb{K}}^{\leq 0}} (A,B).\]
Denote by $dSch_{\K} \subset dSt_{\mathbb{K}}$ the full \bfem{$\infty$-subcategory of  derived $\mathbb{K}$-schemes}, and we simply write $dAff_{\K} \subset dSch_{\mathbb{K}}$ for the full \bfem{$\infty$-subcategory of  affine derived $\mathbb{K}$-schemes.}

\begin{remark}
\begin{enumerate}
            \itemsep=0.2cm
    \item [ ]
    \item     The notation $\spec (-)$ for cdgas must be understood in a  \emph{categorical sense}, contrary to ordinary algebraic geometry, where it actually refers to the prime spectrum of a ring. 
    \item The use of $\spec$-notation is still relevant in geometric terms because the points of a derived scheme correspond to those of its truncation, which is simply an ordinary scheme. That is, for a cdga $A$, the points of $\spec A$ are just the ones in its truncation $\spec (H^0(A))$, which is the ordinary prime spectrum of $H^0(A)\in CAlg_{\mathbb{K}}$, see \cite{ToenHAG} . 
    \item As the points of $\spec A$ are just the ones in $\spec (H^0(A))$, we can consider a derived $\K$-scheme $ X$ as \emph{an infinitesimal thickening of its truncation}.
\end{enumerate}
    \end{remark}

Let us also mention another useful type of derived stacks, the \emph{derived Artin $\K$-stacks}, which generalize derived schemes but still have reasonable and tractable properties. 

Roughly speaking, we call ${ X} \in dSt_{\K}$ a \emph{derived Artin $\K$-stack} if it can be locally represented by an affine derived $\K$-scheme with respect to the ``smooth topology''. Thus, we require the existence of a ``smooth'' surjection $\varphi: U \rightarrow  X$ (of some relative dimension) so that $U$ is a disjoint union of affine derived schemes. Drawing from the foundational framework of To\"en and Vezzosi \cite{ToenHAG}, the concept of geometricity for derived stacks is defined inductively relative to a chosen class of morphisms:

\begin{definition}[{cf. \cite[Definition 1.3.3.1]{ToenHAG}}] \label{def:n_geometric}
\begin{enumerate}
    \item A derived stack $X$ is \bfem{$(-1)$-geometric} if it is representable (i.e., an affine derived $\K$-scheme).
    \item A morphism of derived stacks $f: X \rightarrow Y$ is \bfem{$(-1)$-representable} if for any affine derived $\K$-scheme $U$ and any morphism $U \rightarrow Y$, the homotopy pullback $X \times^h_Y U$ is an affine derived $\K$-scheme.
\end{enumerate}
For $m \ge 0$, we define inductively:
\begin{enumerate}
\setcounter{enumi}{2}
    \item A derived stack $X$ is \bfem{$m$-geometric} if the diagonal morphism $X \rightarrow X \times^h X$ is $(m-1)$-representable, and $X$ admits an \bfem{$m$-atlas} (a smooth surjective morphism $\coprod U_i \rightarrow X$ from a disjoint union of affine derived schemes such that each $U_i \rightarrow X$ is $(m-1)$-representable).
    \item A morphism $f: X \rightarrow Y$ is \bfem{$m$-representable} if for any affine derived scheme $U \rightarrow Y$, the homotopy pullback $X \times^h_Y U$ is an $m$-geometric stack.
\end{enumerate}
\end{definition}

\begin{definition}(Formal-\cite[Definition 1.3.3.1]{ToenHAG})
An object $X \in dSt_{\K}$ is called a \bfem{derived Artin $\K$-stack} if it is $m$-geometric for some integer $m \geq 0$, and the underlying classical stack is 1-truncated. 
\end{definition}

The upshot is that any such object $ X$ of $ dSt_{\K}$ comes with a smooth surjective morphism  $\varphi: U \rightarrow  X$ with $U$ a  derived $\K$-scheme. We call such morphism an \bfem{atlas.} Therefore, the following definition will be sufficient for our purposes.

\begin{definition} (Informal)
    By a \bfem{derived Artin $\K$-stack},  we mean an object $X$ of $ dSt_{\K}$ possessing an atlas (smooth of some relative dimension) near each point of $X$.
\end{definition} 

\paragraph{Good affine models.} In addition to the discussions above, every derived Artin $ \K$-stack is locally modeled on  atlases formed by special types of cdgas.
Let us begin with a relevant concept.

\begin{definition}\label{defn_standard form cdgas}
    We say $ A \in cdga_{\K}^{\leq 0}$ is of \bfem{standard form} if  $A^0$ is a smooth finitely generated $\mathbb{K}$-algebra; the  module $\Omega^1_{A^0}$ of K\"{a}hler differentials is free $A^0$-module of finite rank; and the graded algebra $A$ is freely generated over $A^0$ by finitely many generators, all in negative degrees. 
\end{definition}

If $ A$ is a standard form cdga, then we may assume that there exist $k\in \mathbb{Z}_{\geq 0}$ and generators $x_1^{-i}, x_2^{-i}, \cdots, x_{m_i}^{-i} $ in $A^{-i}$ (after localization, if necessary) with $i= 1, 2, \cdots, k$  \ and $m_i \in \mathbb{Z}_{\geq 0}$ such that 
\begin{equation}
    A = A(0) \big[x_j^{-i} : i= 1, 2, \dots, k, \ j= 1,2, \dots, m_i\big].
\end{equation} In other words, we can consider  $A$ as a \emph{graded polynomial algebra} over $A(0)$ on negatively graded finitely many generators. For more details, see \cite[Example 2.8]{Brav}.

\paragraph{Submersions.} Similarly to the local models for spaces, there is a suitable class of morphisms in $ cdga_{\K}^{\leq 0}$ for doing explicit computations.
 
 \begin{definition}\label{def_submersion}
     A morphism $\beta: A \rightarrow B$ of standard form cdgas is called a \bfem{submersion} if the induced morphism $\beta_*: \Omega_A^1 \otimes_{A} B \rightarrow \Omega_B^1$ is injective in each degree.
 \end{definition}
 
 \begin{remark}
     Submersions are in fact a suitable class of morphisms in $ cdga_{\K}^{\leq 0}$ for doing explicit computations concerning relative cotangent complexes. In this regard, if $\beta: A \rightarrow B$ is a submersion of standard form cdgas, then the module of relative K\"{a}hler differentials $\Omega_{B/A}^1$ is a model for the relative cotangent complex, and hence we take $\mathbb{L}_{B/A}=\Omega_{B/A}^1$. 
 \end{remark}

\subsection{Functions, quasi-coherent sheaves, and perfect complexes} 

When $X=\spec A$ is an affine $\K$-scheme, a function $f$ is simply an element of $A$. If $X$ is a scheme, then a \textit{global function} on $X$ consists of the data of functions on each affine open subset which are compatible on overlaps. This definition can be extended to the case of derived stacks using the limit taken in the proper category. We then define the \bfem{algebra of functions} on $X \in dSt_{\K}$ by \begin{equation}\label{defn_O(X) as inf limit}
    \mathcal{O}(X)= \displaystyle \lim_{A \in cdga_{\K}^{\leq 0}, \ \spec A \rightarrow X} A.
\end{equation} Likewise,  the \bfem{$\infty$-category of quasi-coherent sheaves on $X$} is given by 
\begin{equation}\label{defn_QCoh(X) as inf limit}
    {QCoh}(X)= \displaystyle \lim_{A \in cdga_{\K}^{\leq 0}, \ \spec A \rightarrow X} \mathrm{Mod}_A,
\end{equation}where the limit in (\ref{defn_O(X) as inf limit}) is taken in the $\infty$-category of cdgas and the limit in (\ref{defn_QCoh(X) as inf limit}) in the $\infty$-category of $\infty$-categories, along the pullback functors. 

Within the $\infty$-category of quasi-coherent sheaves $QCoh(X)$, a particularly well-behaved class is that of \emph{perfect complexes}. The foundational properties of perfect modules and relative cotangent complexes were established in the homotopical algebraic geometry framework of To\"en and Vezzosi \cite{ToenHAG}. We extract and recall the relevant statements here to support our subsequent main proofs.

An $A$-module $M$ over a derived affine scheme $\spec A$ is called \bfem{perfect} if it is strongly dualizable. Perfect modules satisfy several important finiteness and descent properties:
\vspace{1in}

\begin{theorem} \label{thm:HAG_perfect_properties} \emph{(Properties of Perfect Modules, \cite{ToenHAG})}
    Let $X$ be a derived Artin $\K$-stack and $\mathcal{E}$ an object in $QCoh(X)$.
    \begin{enumerate}
        \item \label{item:perfect_fp} \emph{(Finitely presented vs. Perfect)} In the stable $\infty$-category of quasi-coherent modules, finitely presented objects are exactly the perfect complexes \emph{\cite[Corollary 1.2.3.8]{ToenHAG}}.
        \item \label{item:perfect_stack} \emph{(Stack property)} The presheaf assigning to each derived affine scheme its $\infty$-category of perfect modules satisfies descent and forms a stack over the smooth (or \'{e}tale) topology \emph{\cite[Corollary 1.3.7.4]{ToenHAG}}.
        \item \label{item:perfect_conservative} \emph{(Strict Conservativity)} Base change functors for perfect complexes along a valid $\tau$-covering family (such as an \'{e}tale or smooth atlas) are strictly conservative \emph{\cite[Corollary 1.3.2.7]{ToenHAG}}.
        \item \label{item:perfect_vanishing} \emph{(Vanishing locus)} If $\mathcal{E}$ is a perfect module on $X$, the locus where $\mathcal{E} \simeq 0$ defines a formal Zariski-open derived substack $U \subseteq X$ \emph{\cite[Proposition 1.2.10.1]{ToenHAG}}. Specifically, a morphism from a derived affine scheme $S$ to $X$ factors through $U$ if and only if the derived pullback of $\mathcal{E}$ to $S$ vanishes.
    \end{enumerate}
\end{theorem}

\subsection{Cotangent and de Rham complexes}
\paragraph{The relative cotangent complex.} The deformation theory of derived stacks and the definition of shifted differential forms rely on the \textit{relative cotangent complex}. For a given morphism $A \rightarrow B$ of cdgas, the \bfem{relative cotangent complex} $\mathbb{L}_{B/A} \in \mathrm{Ho}(B-\mathrm{Mod})$ universally corepresents the space of derived derivations. Globalizing this concept, any morphism of derived Artin stacks $f: X \rightarrow Y$ possesses a relative cotangent complex $\mathbb{L}_{X/Y} \in QCoh(X)$. 
This construction satisfies several fundamental properties that we will extensively use later on:
\begin{theorem} \label{thm:hag_cotangent_properties} \emph{(Properties of Cotangent Complexes, \cite{ToenHAG})}
    \begin{enumerate}
        \item \label{item:cotangent_fp} \bfem{Finite presentation} \emph{\cite[Proposition 2.2.2.4]{ToenHAG}}: If $X$ and $Y$ are locally finitely presented derived Artin $\K$-stacks and $f:X\rightarrow Y$ is locally of finite presentation, then $\mathbb{L}_{X/Y}$ is a coherent object (almost perfect complex) in $QCoh(X)$. If, additionally, $X$ and $Y$ are smooth (or the morphism is perfect), then $\mathbb{L}_{X/Y}$ is perfect.
        \item \label{item:cotangent_transitivity} \bfem{Transitivity sequence} \emph{\cite[Lemma 1.4.1.16]{ToenHAG}}: For composable morphisms of derived stacks $X \xrightarrow{f} Y \xrightarrow{g} Z$, there is a canonical exact triangle (homotopy cofiber sequence) of perfect tangent complexes in $QCoh(X)$:
        \begin{equation}
            f^* \mathbb{L}_{Y/Z} \longrightarrow \mathbb{L}_{X/Z} \longrightarrow \mathbb{L}_{X/Y} \longrightarrow f^* \mathbb{L}_{Y/Z}[1].
        \end{equation}
        \item \label{item:cotangent_basechange} \bfem{Base change} \emph{\cite[Lemma 1.4.1.16]{ToenHAG}}: For any homotopy pullback square of derived stacks
        \begin{equation*}
            \begin{tikzcd}
                X' \arrow[r, "g'"] \arrow[d, "f'"'] & X \arrow[d, "f"] \\
                Y' \arrow[r, "g"'] & Y
            \end{tikzcd}
        \end{equation*}
        the canonical base change morphism $(g')^*\mathbb{L}_{X/Y} \longrightarrow \mathbb{L}_{X'/Y'}$ is an equivalence in $QCoh(X')$. In particular, the derived pullback of the relative cotangent complex $\mathbb{L}_{X/Y}$ to a geometric fiber $X_y$ canonically identifies with the cotangent complex of the fiber: $\iota_y^*\mathbb{L}_{X/Y} \simeq \mathbb{L}_{X_y}$.
    \end{enumerate}
\end{theorem}

\paragraph{The de Rham complex.} When $X=\spec A$ is a derived affine scheme with $A$ a cofibrant cdga (e.g. standard form cdgas with the property that $\mathbb{L}_A\simeq \Omega_A^1$), then its \emph{de Rham complex} is given as $$DR(X)= DR(A)\simeq\mathrm{Sym}_A (\Omega^1_A [-1]),$$ where the graded mixed differential is given by the universal derivation $A\rightarrow \Omega^1_A$ extended via the Leibniz rule. Then for $X$ an arbitrary stack, we define its \bfem{de Rham complex} as \begin{equation}\label{defn_DR(X) as inf limit}
        DR(X)= \lim_{A \in cdga, \ \spec A \rightarrow X} DR(A).
    \end{equation} 
Note that any derived stack $X$ has a graded mixed de Rham complex $DR(X)$. If, in addition, $X$ is Artin and $\mathbb{L}_X$ denotes its cotangent complex, then we have 
\begin{equation}
    \mathcal{A}^p(X,n) \simeq \pi_0 Map_{QCoh(X)}(\mathcal{O}_X, \wedge^p\mathbb{L}_X[n]).
\end{equation}

\subsection{Forms and shifted symplectic structures} 
 We now outline  the theory developed by Pantev, Toën, Vaquié and Vezzosi in \cite{PTVV}.   
\begin{definition}
    A \bfem{$p$-form of degree $n$ on $X$} is an $n$-cocycle in $DR^p(X) [p]$. 
    Likewise,  by a \bfem{closed $p$-form of degree $n$ on $X$}, we mean an $n$-cocycle in $\prod_{i \geq p} DR^i(X) [p]$ for the total differential $d_{tot}=d+ \dR$. Denote the spaces of such forms by $\mathcal{A}^p(X,n)$ and $\mathcal{A}^{p,cl}(X,n)$, respectively.
\end{definition}   

More precisely, $\omega \in \mathcal{A}^p(X,n)$ is a $d$-closed element of $DR(X)$ of weight $p$ and cohomological degree $p+n$. Similarly, $\omega \in \mathcal{A}^{p,cl}(X,n)$ is a $(d+\dR)$-closed element  of $DR(X)$ given as a formal series $\omega_p+\omega_{p+1} + \cdots$ such that $\omega_q$ has weight $q$ and degree $n-(q-p)$. 

Also, we define an \bfem{equivalence relation} by considering paths between objects in $\mathcal{A}^{p , cl}(X,n)$, respectively. A \bfem{path} from $\omega^1$ to $\omega^2$ in $\mathcal{A}^p(X,n)$ is given by an element $h$ of weight $ p $ and degree $ (p + n) -1 $ such that $\omega^1-\omega^2=dh.$ Similarly, a path from $\omega^1=\omega^1_p + \omega^1_{p+1} + \cdots $ to $\omega^2=\omega^2_p + \omega^2_{p+1} + \cdots $ is given by a formal series $h=h_p + h_{p+1} + \cdots $ such that we have $\omega^1_q -\omega^2_q= dh_q + \dR h_{q-1}.$ See \cite{Safronov2019}.

\begin{definition}
    An \bfem{$n$-shifted symplectic structure $\omega$ on $X$} is a closed $2$-form of degree $n$ on $X$, such that the induced morphism $\omega^{\flat} : \mathbb{T}_X \rightarrow \mathbb{L}_X [n]$ is an equivalence (the \emph{non-degeneracy condition}). Denote the space of such by $\mathsf{Symp}(X,n)$.
\end{definition}
\begin{definition} \label{defn_isotropic Lag}
    Let $f: \mathcal{L} \to X$ be a morphism. By an \emph{\textbf{$n$-shifted isotropic structure}} on $f$ (relative to $ \omega_{X} $), we mean a path $h_{\mathcal{L}}$ from $0$ to $f^*(\omega_{X})$ in  $\mathcal{A}^{2, cl} (\mathcal{L}, n)$.
    
    An \bfem{$n$-shifted Lagrangian structure} on  $ f $  is defined to be an $n$-shifted isotropic structure such that the sequence $\mathbb{T}_{\mathcal{L}}\rightarrow f^*(\mathbb{T}_{{X}}) \rightarrow \mathbb{L}_{\mathcal{L}}[n]$ is a homotopy fiber sequence\footnote{Equivalently, we can require the  map $\chi_{h_{\mathcal{L}}}:     \mathbb{T}_{\mathcal{L}/X} \rightarrow \mathbb{L}_{\mathcal{L}}[n-1]$ to be an equivalence.}.  In that case, we  say that $\mathcal{L}$ is \bfem{Lagrangian in $ ({X}, \omega_{X}).$}
\end{definition} 

\paragraph{Some results in derived symplectic geometry} There are many interesting literature results in various flavors, see \cite{Anel2021, Toen2014},  but we will mention only a few. 
\begin{itemize}
		
	\item PTVV's original paper provides some existence results for shifted symplectic structures, such as derived symplectic structures on fiber products of derived stacks and that on mapping stacks \cite[Theorems 2.9 \& 2.5]{PTVV}. 
		For instance, \cite[Theorem 2.9]{PTVV} states that given the diagram of two Lagrangians in an $n$-symplectic derived Artin stack $(F,\omega)$
	\begin{equation}
		\begin{tikzcd}
			& Y \arrow[d, "g"] \\
			X \arrow[r, "f"] & {(F,\omega),}    
		\end{tikzcd}
	\end{equation}the derived Artin stack $X\times_{f,F,g} Y$ carries an $(n-1)$-shifted symplectic structure. 

\item When $X$ is a smooth stack and $f$ is a regular function on $X$, with the differential map $df: X \rightarrow T^*X$, its (derived) critical locus $\mathrm{Crit}(f)$, defined as the fiber product $X\times_{df,T^*X,0} X$ of the zero section with the graph of $df$ inside the ($0$-symplectic) cotangent stack  $T^*X$, has a canonical $(-1)$-shifted symplectic structure \cite{PTVV}.
\item Calaque in \cite{Calaque2019} showed that given a derived Artin $\mathbb{K}$-stack $X$, its shifted cotangent stack $T^*[n] X$ carries a canonical $n$-shifted symplectic
structure.
\item  \cite[Theorem 2.4]{Calaque2019}  proves that the zero section morphism $\iota_X: X \rightarrow T^*[n]X $ has a shifted Lagrangian structure\footnote{A similar result appears in classical symplectic geometry.}.
\item \cite{CPTVV} provides further developments for derived geometry, leading to the formulation of \emph{shifted Poisson structures} on derived Artin stacks. 
\item Safronov \cite{Safronov2023} introduced the notion of \emph{shifted geometric quantization} in the setting of shifted symplectic structures, with a list of inspiring constructions. 

\end{itemize}
\section{Relative structures}
\subsection{Relative shifted symplectic structures and fibrations} \label{sec:Relative structure}

Concepts from PTVV's symplectic geometry \cite{PTVV} can be extended to the
relative setting. Given a morphism $\pi: X \rightarrow S$ of derived Artin stacks, we can define the \bfem{relative de Rham algebra $DR(X/S)$} as a graded commutative dg algebra \begin{equation}
    DR(X/S)\simeq \Gamma(\mathrm{Sym}_{\mathcal{O}_X}(\mathbb{L}_{X/S}[-1])),
\end{equation} where $ \mathbb{L}_{X/S} $ is the \emph{relative cotangent complex} defined via the cofiber sequence $$\pi^*\mathbb{L}_S \rightarrow \mathbb{L}_X \rightarrow \mathbb{L}_{X/S}.$$

When $\pi$ is a map of derived affine schemes $\spec A \rightarrow \spec B$ coming from a cofibration $B\rightarrow A$ of cofibrant cdgas, we have \begin{equation}
DR(X/S)=DR(A/B) \simeq \mathrm{Sym}_A (\Omega^1_{A/B} [-1]).
\end{equation}

\begin{definition} From \cite{CPTVV} (also see \cite{Safronov2023}) we have the following relative notions:

\begin{itemize}
    \item[(i)] 
    A \bfem{relative $p$-form of degree $n$ on $\pi: X \rightarrow S$} is a $d$-closed element of $DR(X/S)$ of weight $p$ and cohomological degree $p+n$.  Denote the space of such forms by $\mathcal{A}^p(X/S,n)$. 

    \item[(ii)] 
    A \bfem{closed relative $p$-form of degree $n$ on $\pi: X \rightarrow S$} is a $(d+\dR)$-closed element $\omega_p+\omega_{p+1} + \cdots$ of $DR(X/S)$ such that $\omega_q$ is a $q$-form of degree $n-(q-p)$.  Denote the space of such forms by $\mathcal{A}^{p,cl}(X/S,n)$.
\end{itemize}   
\end{definition}
From the above definitions, we have:  

\begin{remark} Let $\pi: X \rightarrow S$ be a morphism of derived Artin $\K$-stacks. Then we have:
    \begin{enumerate}
        \item  There exists a null-homotopic sequence $DR(S) \xrightarrow{\pi^*} DR(X) \xrightarrow{\cdot/S} DR(X/S)$ such that the morphism $DR(X) \rightarrow DR(X/S)$ factors as $$DR(X) \rightarrow DR(X \times S /S) \rightarrow DR(X/S),$$ where the first map allows us to see forms on $X$ as relative forms on the projection map $X \times S \rightarrow S;$ and the second one is the pullback along $(id \times \pi): X \rightarrow X \times S. $ See \cite{Calaque2019}.
        
        \item If $\pi$ is locally of finite presentation, we can define $\mathbb{T}_{X/S}= \mathbb{L}_{X/S}^{\vee}$ as the dual (recall the fiber sequence $\mathbb{T}_{X/S} \rightarrow \mathbb{T}_X \rightarrow \pi^* \mathbb{T}_S$). Then a relative 2-form $\omega$ of degree $n$ induces a morphism 
        \begin{equation} \label{eqn:relative symplectic form}
            \omega^{\flat}: \mathbb{T}_{X/S} \rightarrow \mathbb{L}_{X/S} [n].
         \end{equation}   
        \end{enumerate}
\end{remark}

The above observations lead to the following definition:

\begin{definition}\label{defn_relative symplectic strc}
        A \bfem{relative $n$-shifted symplectic structure $\omega$ on $\pi: X \rightarrow S$} is a closed relative $2$-form of degree $n$ on $\pi$ such that the induced map $\omega^{\flat}: \mathbb{T}_{X/S} \rightarrow \mathbb{L}_{X/S} [n]$ is an equivalence.
\end{definition}

\begin{remark}
    Assuming all stacks are $\K$-stacks, the relative objects on the morphism $X \rightarrow \mathrm{pt}$, with target as the point $\mathrm{pt}$, will recover the ordinary (absolute) shifted geometric structures defined previously.
\end{remark}

Next we want to introduce a notion in DSG which will replace the notion of a symplectic fibration in smooth theory. To this end, let $\pi: X \rightarrow S$ be a morphism of derived Artin stacks. Define the \bfem{fiber $ X_s $ over $s\in S$} by 
\begin{equation}
    \begin{tikzcd}
        X_s:=X \times_{S} \{*\} \arrow[r] \arrow[d] & \{*\} \arrow[d, "s"] \\
        X \arrow[r,"\pi"]                                 & S,             
    \end{tikzcd}
\end{equation}where we denote by $\iota_s$ the inclusion of the fiber $X_s \hookrightarrow X$. Note that if $\omega$ is a relative $n$-shifted symplectic structure on $\pi$, then, over $s\in S$, the fiber $X_s$ has the induced $n$-shifted symplectic structure, written $\omega_s$. Thus, we define:

\begin{definition} \label{defn_shifted symp fibration}
    An \bfem{$n$-shifted symplectic fibration} is a morphism  $\pi: X \rightarrow S$ of derived Artin $\K$-stacks which admits a relative $n$-shifted symplectic structure $\omega$ in the sense of Definition \ref{defn_relative symplectic strc}. Alternatively, a morphism  $\pi: X \rightarrow S$ is said to have an \bfem{$n$-shifted symplectic fibration structure} if it has a relative $n$-shifted symplectic structure.
\end{definition}

As we will see in the proof of Theorem \ref{thm:2} below and by the following result from DSG literature, this notion of $n$-shifted symplectic fibration is a convenient replacement of smooth ones.

\begin{proposition} \label{prop: way of inducing symp fib structure} \emph{\cite[Corollary 2.3.8]{Park2025}, \cite[Proposition 1.10] {Safronov2023}.}
    Let $Y$ be an $n$-shifted symplectic derived Artin $\K$-stack. Given a Lagrangian $f:X\rightarrow Y$ and a Lagrangian fibration $p:Y\rightarrow S$, the composition $p\circ f: X \rightarrow S$ has an $(n-1)$-shifted symplectic fibration structure.
\end{proposition}

\pf Let $s\in S$. As $p:Y\rightarrow S$ has a Lagrangian fibration structure, each fiber inclusion $Y_s \hookrightarrow Y$ carries a Lagrangian structure, and hence we get the diagram 
\[\begin{tikzcd}
X \times_{Y} Y_s \arrow[r, dashed] \arrow[d]              & Y_s \arrow[d, "i-Lag"'] \arrow[r] & \{*\} \arrow[d, "s"] \\
X \arrow[r, "f-Lag"] \arrow[rr, "p\circ f", bend right] & Y \arrow[r, "p"]                  & S                   
\end{tikzcd}\] Applying PTVV's Lagrangian intersection theorem to the left-hand square, we conclude that the geometric fiber $X \times_{Y} Y_s$ has an $(n-1)$-shifted symplectic structure. However, showing that geometric fibers are symplectic is not sufficient to construct a global relative form. The existence of the globally defined relative closed 2-form $\omega_{rel} \in \mathcal{A}^{2,cl}(X/S, n-1)$ governing the $(n-1)$-shifted symplectic fibration structure follows directly from the construction of isotropic structures globally over the base $S$, as detailed in \cite[Proposition 1.10]{Safronov2023}. The relative form so constructed restricts to the canonical symplectic structures on the fibers. Thus, the composition $p\circ f$ has an $(n-1)$-shifted symplectic fibration structure.

\epf

Here are some examples as applications of Proposition \ref{prop: way of inducing symp fib structure}:

\begin{example} \label{example:cotangent stack}
    Suppose that $X$ is a derived Artin $\K$-stack locally of finite presentation and $\pi_X: T^*X\rightarrow X$ is the natural projection. It has been shown in \cite{Calaque2019} that the cotangent stack $T^*X$ has a natural Liouville 1-form $\lambda \in \mathcal{A}^1(T^*X, 0)$ such that 0-shifted symplectic structure on $T^*X$ is given by $\omega:=d_{dR}\lambda$ and that the map $\pi_X$ carries a natural structure of an 0-shifted Lagrangian fibration meaning that, for any $x\in X$, the inclusion of the fiber $\iota_x:X_x\to T^*X$ has a 0-shifted Lagrangian structure. Therefore, we have the diagram
\[
    \begin{tikzcd}
X_x \arrow[r, "\iota_x",  hook] \arrow[rd, "{\pi_X\circ\iota_x}"', dashed] & T^*X \arrow[d, "\pi_X"] \\
       & X.                      
\end{tikzcd}
    \] Here, the vertical map is an 0-shifted Lagrangian fibration, and the horizontal one has a  $0$-shifted Lagrangian structure. By  Proposition \ref{prop: way of inducing symp fib structure}, the composition $ \pi_X\circ\iota_x: X_x \rightarrow X$  carries a $(-1)$-shifted symplectic fibration structure.
    
    Note that this example can be extended verbatim to the case of $n$-shifted cotangent $\K$-stacks $T^*[n]X$, implying that the composition $X_x \hookrightarrow T^*[n]X \rightarrow X$ will  have an $(n-1)$-shifted symplectic fibration structure.
\end{example}

\begin{example} \label{example:β-twisted cotangent stack}
        Recall that  if $ X $ is a derived Artin $\K$-stack locally of finite presentation equipped with a closed 1-form $\beta $ of degree $ (n + 1) $, we define the \emph{\textbf{$ n $-shifted $\beta$-twisted cotangent $\K$-stack of $ X $}} to be the fiber product
\begin{equation}
    \begin{tikzcd}
{T^*_{\beta} [n] X} \arrow[d] \arrow[r] & X \arrow[d, "\Gamma_{0}"] \\
X \arrow[r, "\Gamma_{\beta}"]           & {T^*[n+1]X}.      \end{tikzcd}
\end{equation}Here, $\Gamma_{\beta},\Gamma_{0}: X \rightarrow T^*[n+1]X$ are the graphs of $\beta$ and $0$-section, respectively. Note that from \cite[Theorem 2.4]{Calaque2019} and \cite[Definition 1.16 and Remark 1.17]{Safronov2023}, both morphisms $ \Gamma_{\beta},\Gamma_{0} $ have $(n+1)$-shifted Lagrangian structures, and hence the resulting fiber product $ T^*_{\beta} [n] X $ is an $n$-shifted symplectic $\K$-stack. Moreover, \cite[Proposition 1.21]{Safronov2023} shows that the projection $T^*_{\beta} [n] X \rightarrow X$ carries a natural structure of an $n$-shifted Lagrangian fibration.
Then, from Proposition \ref{prop: way of inducing symp fib structure},  the composition $X_x \hookrightarrow T^*_{\beta}[n]X \rightarrow X$   admits an $(n-1)$-shifted symplectic fibration structure.

In particular, assume that $X$ is a smooth $\K$-scheme and $f : X \rightarrow A^1$ a regular function. As the 1-form $d_{dR}f$ (of $\deg$ 0) is closed,  we get $T^*_{d_{dR}f}[-1]X$, the $(-1)$-shifted twisted cotangent bundle, also called the \bfem{derived critical locus} of $f.$ Then the composition $X_x \rightarrow T^*_{d_{dR}f}[-1]X \rightarrow X $ is a (-2)-shifted symplectic fibration.
\end{example}

\begin{example} \label{example:conormal stack}
    Given $f: X \rightarrow S$ a morphism of derived Artin $\K$-stacks locally of finite presentation, we define the \bfem{$n$-shifted conormal $\K$-stack} $N^*[n] f \longrightarrow X$ to be the total space of the shifted relative cotangent complex $\mathbb{L}_{X/S}[n-1]$. Then the natural fiber  sequence $f^*\mathbb{L}_S \rightarrow \mathbb{L}_X \rightarrow \mathbb{L}_{X/S}$ gives rise to a morphism on the level of total spaces \cite{Safronov2019} \[N^*[n] f \longrightarrow T^*[n]S. \] Calaque in \cite{Calaque2019} shows that the morphism $N^*[n] f \longrightarrow T^*[n]S$ carries an $n$-shifted Lagrangian structure, giving the  diagram   
     \[
    \begin{tikzcd}
N^*[n]f \arrow[r] \arrow[rd,  dashed] & T^*[n]S \arrow[d, "\pi_S"] \\
       & S                    
\end{tikzcd}
    \]where the vertical map is an $n$-shifted Lagrangian fibration. As for previous examples, Proposition \ref{prop: way of inducing symp fib structure} implies that  the composition $N^*[n]f   \longrightarrow  S$ carries the structure of an $(n-1)$-shifted symplectic fibration.
    Note that the result can also be extended verbatim to the morphisms $N_{\beta}^*[n] f \longrightarrow T_{\beta}^*[n]S$ of twisted $\K$-stacks. See \cite[Definition 1.22] {Safronov2023}.
\end{example}

\begin{example} \label{ex:mapping_stack} 
(See also \cite[Proposition 2.3.4]{Park2025}) Under certain conditions, one can lift Lagrangian and Lagrangian fibration structures on mapping stacks \cite[Theorem 2.10]{Calaque2015}. Let $Y$ be an $n$-shifted symplectic derived Artin $\K$-stack. Suppose we are given a Lagrangian $f:L\rightarrow Y$ and a Lagrangian fibration $p:Y\rightarrow S$. If $X$ is a smooth projective Calabi-Yau $m$-fold, then we get the induced diagram \[
    \begin{tikzcd}
\mathsf{Map}(X,L) \arrow[r, "f \,\circ \,-"] \arrow[rd,  dashed] & \mathsf{Map}(X,Y) \arrow[d, "p \,\circ\,- "] \\
       & \mathsf{Map}(X,S)                     
\end{tikzcd}
   \]where the vertical map is an $(n-m)$-shifted Lagrangian fibration; the horizontal one is $(n-m)$-shifted Lagrangian. From Proposition \ref{prop: way of inducing symp fib structure}, the composition $\mathsf{Map}(X,L) \longrightarrow \mathsf{Map}(X,S)$ is then an $(n-m-1)$-shifted symplectic fibration.
\end{example}

\begin{example} \label{ex:moment map}
    Denote by $ BG  $ the \emph{classifying stack} of an affine algebraic Lie group $ G $ equipped with a nondegenerate invariant symmetric bilinear pairing $ \langle -, - \rangle $ on its Lie algebra $\mathfrak{g}$. More precisely, it is defined as the \emph{quotient stack}
\begin{equation}
    BG:= */G= \mathrm{colim} \bigg(*\mathrel{\substack{\textstyle\leftarrow\\[-0.1ex]
            \textstyle\leftarrow \\[-0.1ex]}}  G
    \mathrel{\substack{\textstyle\leftarrow\\[-0.1ex]
            \textstyle\leftarrow \\[-0.1ex]
            \textstyle\leftarrow\\[-0.3ex]}} G \times G
    \mathrel{\substack{\textstyle\leftarrow\\[-0.1ex]
            \textstyle\leftarrow \\[-0.1ex]
            \textstyle\leftarrow \\[-0.1ex]
            \textstyle\leftarrow \\[-0.3ex]}} 
    \cdot\cdot\cdot \bigg),
\end{equation}  where the maps are given by the action and projection. Note that $ BG $ carries a canonical 2-shifted symplectic structure $\omega_{BG}$ \cite{PTVV}. 

Let $X$ be a $2$-shifted symplectic derived Artin stack equipped with a smooth action of $G$, and let $\mu: X \rightarrow \mathfrak{g}^*[2]$ be a \emph{$2$-shifted moment map}, meaning that the map $\mu$ is $G$-equivariant and there is a $3$-shifted Lagrangian structure on the map \[ \mu/G: X/G \longrightarrow \mathfrak{g}^*[2]/G.\]

There also exists an equivalence $X/G \times_{\mathfrak{g}^*[2]/G} \mathfrak{g}^*[2] \simeq X$ of $2$-shifted symplectic stacks \cite[Definition 4.3.2]{GrataloupPHD}. Here $ \mathfrak{g}^*[2]/G \simeq T^*[3]BG$ as $3$-shifted symplectic stacks \cite[Section 2]{AnelCalaque2022}. We then get the diagram     \[
    \begin{tikzcd}
X/G \arrow[r] \arrow[rd,  dashed] & \mathfrak{g}^*[2]/G \simeq T^*[3]BG \arrow[d, "\pi_{BG}"] \\
       & (BG, \omega_{BG})                    
\end{tikzcd}
    \]where the vertical map is a $3$-shifted Lagrangian fibration. As for previous examples, Proposition \ref{prop: way of inducing symp fib structure} implies that  the composition $X/G \longrightarrow (BG, \omega_{BG})$ carries a $2$-shifted symplectic fibration structure.
\end{example}

\subsection{Compatible shifted symplectic fibrations}\label{sec:Ind_symp_fib}

We will now establish the first main result, Theorem \ref{thm:2}, of this paper; namely: \textit{Given a morphism $\pi: X \rightarrow S$ of derived Artin $\K$-stacks, if $X$ admits an $n$-shifted symplectic structure $\omega_X$ whose restriction to each geometric fiber $X_s$ defines an $n$-shifted symplectic structure on $X_s$, then $\pi$ admits an $n$-shifted symplectic fibration structure compatible with $\omega_X$.} 

Before proving Theorem \ref{thm:2}, let us first introduce the appropriate \bfem{compatibility} notion for shifted symplectic fibrations in the DSG framework:

\begin{definition}[Compatibility with a shifted symplectic fibration] \label{def:compatibility}
Let $\pi: X \rightarrow S$ be a morphism of derived Artin $\K$-stacks with an $n$-shifted symplectic fibration structure $\omega$, and for $s\in S(\K)$,  let $\omega_s := \iota_s^*\omega$ be the induced $n$-shifted symplectic structure on the fiber $X_s$, where $\iota_s: X_s \rightarrow X$ is the base change of $s: \spec\, \K \rightarrow S$. 

We say that a closed $2$-form $\Omega \in \mathcal{A}^{2,cl}(X,n)$ is \bfem{compatible with} $\omega$ if $\iota_s^*\Omega \sim \omega_s$ in $\mathcal{A}^{2,cl}(X_s,n)$ for each $s\in S(\K)$. If $\omega_X$ is an $n$-shifted symplectic structure on $X$ compatible with $\omega$, we also say that $\omega$ is compatible with $\omega_X$, and that $\omega_X$ is compatible with $\pi$ when $\omega$ is clear from the context. For an open derived substack $U \subseteq X$, we use the same terms for $\pi|_U: U \rightarrow S$ with the restriction of $\omega$, whose fibers are $U_s := U \times_X X_s$.
    
\end{definition}

\begin{proof}[\textbf{Proof of Theorem \ref{thm:2}}]
We are given a morphism of derived Artin $\K$-stacks $\pi: X \rightarrow S$ and an absolute $n$-shifted symplectic structure $\omega_X \in \mathsf{Symp}(X, n)$. 

By the definition of relative forms \cite[Definition 1.1]{Safronov2023}, the canonical projection $DR(X) \xrightarrow{\cdot/S} DR(X/S)$ maps the absolute $n$-shifted closed 2-form $\omega_X \in \mathcal{A}^{2,cl}(X, n)$ to a relative $n$-shifted closed 2-form $$\omega_{rel}:=\omega_X/S \in \mathcal{A}^{2,cl}(X/S, n).$$ 

To show that $\omega_{rel}$ defines an $n$-shifted symplectic fibration structure, we must prove that the induced morphism of perfect complexes 
\begin{equation}
    \omega_{rel}^\flat : \mathbb{T}_{X/S} \longrightarrow \mathbb{L}_{X/S}[n] 
\end{equation}
is an equivalence in the stable $\infty$-category $QCoh(X)$.

The relative morphism $\omega_{rel}^\flat$ naturally factors as the composition
\begin{equation}
    \mathbb{T}_{X/S} \xrightarrow{i} \mathbb{T}_X \xrightarrow{\omega_X^\flat} \mathbb{L}_X[n] \xrightarrow{i^\vee[n]} \mathbb{L}_{X/S}[n]
\end{equation}
where $i: \mathbb{T}_{X/S} \to \mathbb{T}_X$ is the canonical inclusion obtained by dualizing the relative cotangent exact triangle (Theorem \ref{thm:hag_cotangent_properties}-(\ref{item:cotangent_transitivity})).

For any $\K$-point $s \in S(\K)$, let $\iota_s: X_s \hookrightarrow X$ denote the inclusion of the geometric fiber. By the base change property for relative cotangent complexes (Theorem \ref{thm:hag_cotangent_properties}-(\ref{item:cotangent_basechange})), the derived pullback of the relative cotangent complex to the fiber canonically identifies with the cotangent complex of the fiber: $\iota_s^*\mathbb{L}_{X/S} \simeq \mathbb{L}_{X_s}$. Consequently, the pullback of the relative morphism $\omega_{rel}^\flat$ to the fiber $X_s$ is exactly the morphism induced by the restricted form:\begin{equation}
    \iota_s^*(\omega_{rel}^\flat) \simeq (\iota_s^*\omega_X)^\flat : \mathbb{T}_{X_s} \longrightarrow \mathbb{L}_{X_s}[n].
\end{equation}

By hypothesis, the restriction $\iota_s^*\omega_X$ is an $n$-shifted symplectic structure on $X_s$. This means that for every geometric point $s \in S(\K)$, the induced map $(\iota_s^*\omega_X)^\flat$ is an equivalence. Therefore, the relative morphism $\omega_{rel}^\flat$ is an equivalence on all geometric fibers of $\pi$.

Since $X$ and $S$ are derived Artin stacks locally of finite presentation, $\pi$ is locally of finite presentation, so its relative cotangent complex $\mathbb{L}_{X/S}$ is almost perfect (coherent), see Theorem \ref{thm:hag_cotangent_properties}-(\ref{item:cotangent_fp}). Since the restriction $\iota_s^*\omega_X$ is an $n$-shifted symplectic structure on $X_s$, it follows that the cotangent complex of each geometric fiber $\mathbb{L}_{X_s} \simeq \iota_s^*\mathbb{L}_{X/S}$ is perfect. Also an almost perfect complex whose derived pullback to all geometric points is perfect is globally a perfect complex. Thus, $\mathbb{L}_{X/S}$ (and its dual $\mathbb{T}_{X/S}$) are perfect complexes in $QCoh(X)$. 

Because perfect modules form a stack (Theorem \ref{thm:HAG_perfect_properties}-(\ref{item:perfect_stack})), checking that a morphism of perfect complexes is an equivalence can be done locally over a valid $\tau$-covering family. Since $X$ is a derived Artin stack, we can pick a concrete covering family, namely a smooth atlas $\coprod_i U_i \to X$ where each chart $U_i \simeq \spec A_i$ is an affine derived $\K$-scheme. By Theorem \ref{thm:HAG_perfect_properties}-(\ref{item:perfect_conservative}), the base change functor along this smooth covering family is strictly conservative \cite[Corollary 1.3.2.7]{ToenHAG}. This guarantees that $\omega_{rel}^\flat$ is a global equivalence in $QCoh(X)$ if and only if its derived pullback to each chart $U_i$ is an equivalence in $QCoh(U_i)$.

On each chart $U_i$, let $K_i$ be the mapping cone (homotopy cofiber) of the restricted morphism $\omega_{rel}^\flat\vert_{U_i}$. By \cite[Proposition 1.2.3.7]{ToenHAG}, $K_i$ is a perfect $A_i$-module. Applying the "Vanishing Locus" property (Theorem \ref{thm:HAG_perfect_properties}-(\ref{item:perfect_vanishing})), the locus where $K_i \simeq 0$ defines a formal Zariski-open derived substack $V_i \subseteq U_i$. By hypothesis, $\omega_{rel}^\flat$ restricts to a quasi-isomorphism on every geometric fiber $X_s$. Because every geometric point $x \in U_i(\K)$ maps to a point in some geometric fiber $X_{\pi(x)}$, the restriction of $\omega_{rel}^\flat$ to any geometric point of $U_i$ is an equivalence. Thus, the vanishing locus $V_i$ contains all geometric points of $U_i$. Topologically, a Zariski-open subset of an affine scheme that contains all its geometric points must be the entire scheme, meaning $V_i = U_i$. Therefore, $K_i \simeq 0$ everywhere on the atlas. By the strict conservativity of the base change functor along the smooth atlas $\coprod_i U_i \to X$, this guarantees that $\omega_{rel}^\flat$ is a global equivalence in $QCoh(X)$ as desired.

We conclude that $\omega_{rel}$ is a relative $n$-shifted symplectic structure on $\pi$. By definition, this implies $\pi$ is an $n$-shifted symplectic fibration. This structure is compatible with $\omega_X$ in the sense of Definition \ref{def:compatibility}. Indeed, $\omega_{rel}$ is the image of $\omega_X$ under the projection to relative forms, and the restriction of forms to a fiber $X_s$ factors as $DR(X) \rightarrow DR(X/S) \rightarrow DR(X_s)$, the second map being the base change along $s: \spec\, \K \rightarrow S$ used in the proof of \cite[Corollary 1.7]{Safronov2023}. So $\iota_s^*\omega_{rel} \sim \iota_s^*\omega_X$ for each $s \in S(\K)$.
\end{proof}

\section{Absolute symplectic structures via symplectic fibrations} \label{sec:results}

In this section, we state the precise version of Theorem \ref{thm:1} and prove it by providing an analogous construction used in the aforementioned theorem of Thurston:

\begin{theorem}[Derived Thurston Theorem for Artin Stacks] \label{thm: proof_thurston in dsg}
    Let $X$ and $S$ be derived Artin $\K$-stacks (both locally of finite presentation and quasi-compact) over a field $\K$ of characteristic zero. Let $\omega \in \mathcal{A}^{2,cl}(X/S, n)$ be an $n$-shifted symplectic fibration structure on $\pi: X \rightarrow (S,\sigma)$, with $\sigma\in \mathsf{Symp}(S,n)$. 
    
    Assume there is a closed $2$-form $\Omega \in \mathcal{A}^{2,cl}(X,n)$ compatible with $\omega$ in the sense of Definition \ref{def:compatibility}, that is, $\iota_s^*\Omega \sim \omega_s := \iota_s^*\omega$ in $\mathcal{A}^{2,cl}(X_s,n)$ for each $s\in S(\K)$. 
    
    Then there exists an absolute $n$-shifted symplectic structure $\omega_X \in \mathsf{Symp}(U,n)$ defined on a dense Zariski-open derived substack $U \subseteq X$ such that:
    \begin{enumerate}
        \item[(i)] $\omega_X \sim \Omega + c \cdot \pi^* \sigma$ \, in $\mathcal{A}^{2,cl}(U,n)$ for a generic constant $c \in \K$, and
        \item[(ii)] the closed $2$-form $\Omega + c \cdot \pi^*\sigma \in \mathcal{A}^{2,cl}(X,n)$ is compatible with $\omega$ in the sense of Definition \ref{def:compatibility}. In particular $\iota_s^*\omega_X \sim \omega_s|_{U_s}$ in $\mathcal{A}^{2,cl}(U_s,n)$ for each $s\in S(\K)$, where $U_s := U \times_X X_s$ and $\iota_s: U_s \rightarrow U$, so $\omega_X$ is compatible with $\pi|_U$.
   \end{enumerate}
   Furthermore, if $X$ is proper over $\K$, then $U = X$, and the structure extends globally.
\end{theorem}

\begin{remark}
    It should be noted that when the map $\pi: X \rightarrow (S,\sigma)$ in Theorem \ref{thm: proof_thurston in dsg} is actually a morphism of \textit{derived $\K$-schemes}, the proof can benefit from the construction of a local representative for the morphism $\pi$ via a submersion of standard form cdgas $\spec B \to \spec A$ \cite[Section 4]{Brav}. This approach applies strictly to derived schemes and certain gluing arguments that mimic the original topological proof of Thurston's theorem \cite[Theorem 6.1.4]{Mcduff}. Note that, in derived algebraic geometry, the conventional \emph{partitions of unity} do not exist. Instead, we must rely on the limit in the $\infty$-category of cdgas to glue local forms by equivalences, see (\ref{defn_DR(X) as inf limit}). 

The following setup, on the other hand, provides a coordinate-free proof for arbitrary derived Artin stacks while preserving the geometric approach. 
\end{remark}

\begin{proof}[\textbf{Proof of Theorem \ref{thm: proof_thurston in dsg}}]
The proof proceeds within the stable $\infty$-category of quasi-coherent sheaves, ${QCoh}(X)$. This approach generalizes Thurston's result to derived Artin stacks without relying on local submersive cdga models for the morphism $\pi$. The proof consists of 4 steps.

\medskip \noindent \textbf{Step 1: Definition of $\omega_X$ and fiber compatibility.} \
By assumption, we are given a global absolute $n$-shifted 2-form $\Omega \in \mathcal{A}^{2,cl}(X,n)$ such that its restriction to any geometric fiber $X_s$ satisfies $\iota_s^*\Omega \sim \omega_s$ in $\mathcal{A}^{2,cl}(X_s,n)$. We define the candidate absolute $n$-shifted symplectic structure on $X$ directly as
\begin{equation} \label{eq:omega_candidate}
    \omega_X := \Omega + c \cdot \pi^*\sigma \quad \in \mathcal{A}^{2,cl}(X, n)
\end{equation}
where $c \in \K$ is a scalar constant (to be determined later). As shown in \cite[Definition 1.13]{PTVV}, the space of closed 2-forms $\mathcal{A}^{2,cl}(X, n)$ is equivalent to the Dold-Kan space associated to the weight-2 piece of the shifted negative cyclic complex $NC^w(X/\K)[n-2](2)$. Since this is a strict complex of $\K$-modules, scalar multiplication by $c \in \K$ is a well-defined operation on the underlying complex that canonically preserves closure. Since $\Omega$ and $\pi^*\sigma$ are closed $n$-shifted 2-forms, $\omega_X$ is also a closed $n$-shifted 2-form.

We first verify the compatibility condition $(ii)$. For any $\K$-point $s \in S(\K)$, the geometric fiber $X_s$ is defined by the following Cartesian diagram in the $\infty$-category of derived stacks $dSt_{\K}$:
\begin{equation} \label{cd:fiber_pullback}
\begin{tikzcd}[column sep=large, row sep=normal]
    X_s \arrow[r, "\iota_s", hook] \arrow[d, "p_s"'] & X \arrow[d, "\pi"] \\
    \spec\, \K \arrow[r, "s"'] & S.
\end{tikzcd}
\end{equation}
By commutativity, the composition $\pi \circ \iota_s$ factors through the constant map to $s$. Since the shifted cotangent complex of the point $\spec \K$ vanishes over $\K$, the derived pullback of the base form to the fiber evaluates identically to zero in $\mathcal{A}^{2,cl}(X_s, n)$. Thus, $\iota_s^*\pi^*\sigma \simeq 0$, and so evaluating the pullback of $\omega_X$ to the fiber gives
\begin{equation} \label{eq:fiber_eval}
    \iota_s^*\omega_X = \iota_s^*\Omega + c \cdot \iota_s^*\pi^*\sigma \sim \omega_s + 0 = \omega_s.
\end{equation}
This verifies the condition $(ii)$ for any choice of the constant $c$. 

\medskip \noindent \textbf{Step 2: Homological splitting of the tangent complex.} \
It remains to prove that for a suitable choice of $c$, the form $\omega_X$ is non-degenerate. By \cite[Definition 1.18]{PTVV}, the non-degeneracy requires that the induced morphism of perfect complexes $\omega_X^\flat: \mathbb{T}_X \to \mathbb{L}_X[n]$ is an equivalence in the homotopy category $\mathrm{Ho}({QCoh}(X))$. 

By dualizing the canonical relative cotangent sequence for stacks (Theorem \ref{thm:hag_cotangent_properties}-(\ref{item:cotangent_transitivity})), the morphism $\pi: X \to S$ induces a canonical exact triangle of perfect tangent complexes in $QCoh(X)$:
\begin{equation} \label{eq:tangent_triangle}
    \begin{tikzcd}
        \mathbb{T}_{X/S} \ar[r, "i"] & \mathbb{T}_X \ar[r, "p"] & \pi^*\mathbb{T}_S \ar[r, "+1"] & \mathbb{T}_{X/S}[1].
    \end{tikzcd}
\end{equation}
The absolute form $\Omega \in \mathcal{A}^{2,cl}(X,n)$ induces the morphism $\Omega^\flat: \mathbb{T}_X \to \mathbb{L}_X[n]$. We define its relative projection onto the vertical tangent complex as $\Omega_{rel}^\flat := i^\vee{[n]} \circ \Omega^\flat \circ i : \mathbb{T}_{X/S} \to \mathbb{L}_{X/S}[n]$, factoring as follows.
\begin{equation} \label{cd:theta_projection}
\begin{tikzcd}
    \mathbb{T}_{X/S} \arrow[r, "i"] \arrow[d, "\Omega_{rel}^\flat"'] & \mathbb{T}_X \arrow[d, "\Omega^\flat"] \\
    \mathbb{L}_{X/S}[n] & \mathbb{L}_X[n] \arrow[l, "i^\vee{[n]}"'].
\end{tikzcd}
\end{equation}
By hypothesis, $\iota_s^*\Omega \sim \omega_s$. Since $\omega_s$ is an $n$-shifted symplectic structure on the fiber $X_s$, the induced map $\omega_s^\flat$ is an equivalence. Because the pullback of the relative cotangent complex to the geometric fiber is the cotangent complex of the fiber (i.e., $\iota_s^*\mathbb{L}_{X/S} \simeq \mathbb{L}_{X_s}$), the restriction of $\Omega_{rel}^\flat$ to any geometric fiber $X_s$ over $s \in S(\K)$ is a quasi-isomorphism. 

Since $X$ and $S$ are locally of finite presentation, the complexes $\mathbb{T}_{X/S}$ and $\mathbb{L}_{X/S}[n]$ are perfect in $QCoh(X)$. Define $K \in QCoh(X)$ as the mapping cone (homotopy cofiber) of the relative morphism $\Omega_{rel}^\flat: \mathbb{T}_{X/S} \to \mathbb{L}_{X/S}[n]$. By \cite[Proposition 1.2.3.7]{ToenHAG}, since both $\mathbb{T}_{X/S}$ and $\mathbb{L}_{X/S}[n]$ are perfect, their mapping cone $K$ is also a perfect complex. 

We then apply \cite[Proposition 1.2.10.1]{ToenHAG} to the perfect complex $K$. This proposition guarantees that the vanishing locus where $K \simeq 0$ constitutes a formal derived Zariski-open immersion $V \hookrightarrow X$. To conclude that $V = X$ globally, we argue topologically over a concrete valid covering family. Choose a smooth atlas $\coprod_i U_i \to X$ where each $U_i$ is an affine derived $\K$-scheme. As shown above, $\Omega_{rel}^\flat$ restricts to a quasi-isomorphism on every geometric fiber $X_s$. Consequently, for any geometric point in any chart $U_i$, the pullback of $K$ vanishes. Topologically, a Zariski-open subset of an affine scheme that contains all its geometric points must be the entire scheme. Hence, the condition $K \simeq 0$ holds over the valid atlas $\coprod_i U_i$. By the strict conservativity of the base change functor along this atlas (\cite[Corollary 1.3.2.7]{ToenHAG}), $K \simeq 0$ globally on $X$. Consequently, $\Omega_{rel}^\flat$ is a global equivalence in $QCoh(X)$.

Because $\Omega_{rel}^\flat$ is an equivalence, we construct a canonical homological retraction morphism $\rho: \mathbb{T}_X \to \mathbb{T}_{X/S}$ via:
\begin{equation} \label{eq:retraction}
    \rho := (\Omega_{rel}^\flat)^{-1} \circ i^\vee{[n]} \circ \Omega^\flat.
\end{equation}
Composing this retraction with the inclusion $i$ yields:
\begin{equation} \label{eq:retraction_comp}
    \rho \circ i \simeq (\Omega_{rel}^\flat)^{-1} \circ (i^\vee{[n]} \circ \Omega^\flat \circ i) \simeq (\Omega_{rel}^\flat)^{-1} \circ \Omega_{rel}^\flat \simeq \operatorname{id}_{\mathbb{T}_{X/S}}.
\end{equation}

In a stable $\infty$-category, the existence of a left inverse for a morphism in an exact triangle implies that the triangle splits \cite[Theorem 1.1.2.14]{LurieHA}. Let $\eta: \pi^*\mathbb{T}_S \to \mathbb{T}_X$ denote the canonical section of $p$ induced by this splitting, so that $p \circ \eta \simeq \operatorname{id}_{\pi^*\mathbb{T}_S}$ and $\rho \circ \eta \simeq 0$. That is, we get:
\begin{equation} \label{cd:triangle_split}
\begin{tikzcd}[column sep=huge]
    \mathbb{T}_{X/S} \arrow[r, "i", shift left=1ex] & \mathbb{T}_X \arrow[l, "\rho", shift left=1ex, dashed] \arrow[r, "p", shift left=1ex] & \pi^*\mathbb{T}_S \arrow[l, "\eta", shift left=1ex, dashed].
\end{tikzcd}
\end{equation}
Thus, in $\mathrm{Ho}({QCoh}(X))$, we obtain the canonical equivalence:
\begin{align} \label{eq:splitting}
    \mathbb{T}_X &\simeq \mathbb{T}_{X/S} \oplus \pi^*\mathbb{T}_S .
\end{align}

\noindent \textbf{Step 3: Block-Diagonalization.} \
Because the homotopy category $\mathrm{Ho}({QCoh}(X))$ is a triangulated (and thus additive) category \cite[Theorem 1.1.2.14]{LurieHA}, morphisms between direct sums can be represented as block matrices. With respect to the splitting (\ref{eq:splitting}), we evaluate the components of $\Omega^\flat: \mathbb{T}_X \to \mathbb{L}_X[n]$. 

The top-left block corresponds to the vertical component, $i^\vee{[n]} \circ \Omega^\flat \circ i = \Omega_{rel}^\flat$. The top-right block describes the cross term mapping $\pi^*\mathbb{T}_S$ to $\mathbb{L}_{X/S}[n]$, given by $i^\vee{[n]} \circ \Omega^\flat \circ \eta$. By the definition of the retraction $\rho$ in (\ref{eq:retraction}), we have the equivalence $i^\vee{[n]} \circ \Omega^\flat \simeq \Omega_{rel}^\flat \circ \rho$. Applying this to the section $\eta$, we find:
\begin{equation} \label{eq:cross_term1}
    i^\vee{[n]} \circ \Omega^\flat \circ \eta \simeq \Omega_{rel}^\flat \circ (\rho \circ \eta) \simeq \Omega_{rel}^\flat \circ 0 \simeq 0.
\end{equation}
The bottom-left block maps $\mathbb{T}_{X/S}$ to $\pi^*\mathbb{L}_S[n]$, given by $\eta^\vee{[n]} \circ \Omega^\flat \circ i$. Because $\Omega$ is an $n$-shifted 2-form, the associated morphism $\Omega^\flat$ satisfies $(\Omega^\flat)^\vee{[n]} \simeq (-1)^{n+1} \Omega^\flat$ in $\mathrm{Ho}({QCoh}(X))$ (cf. \cite[Definition 1.12]{PTVV}). Taking the dual of the top-right block, we obtain:
\begin{equation} \label{eq:cross_term2}
    \eta^\vee{[n]} \circ \Omega^\flat \circ i \simeq (-1)^{n+1} (\eta^\vee{[n]} \circ (\Omega^\flat)^\vee{[n]} \circ i) \simeq (-1)^{n+1} (i^\vee{[n]} \circ \Omega^\flat \circ \eta)^\vee{[n]} \simeq 0.
\end{equation}
Geometrically, $\rho$ is the projection onto the vertical tangent space along its $\Omega$-orthogonal complement, and the two cross terms above vanish in the homotopy category. As a result, $\Omega^\flat$ is block-diagonal:
\begin{equation} \label{diag:direct_sum_splitting}
\begin{tikzcd}[column sep=huge, row sep=large, ampersand replacement=\&]
    \mathbb{T}_X \arrow[r, "\sim"] \arrow[d, "\Omega^\flat"'] \& \mathbb{T}_{X/S} \oplus \pi^*\mathbb{T}_S \arrow[d, "{\begin{pmatrix} \Omega_{rel}^\flat & 0 \\ 0 & C \end{pmatrix}}"] \\
    \mathbb{L}_X[n] \arrow[r, "\sim"] \& \mathbb{L}_{X/S}[n] \oplus \pi^*\mathbb{L}_S[n]
\end{tikzcd}
\end{equation}
Here $C = \eta^\vee{[n]} \circ \Omega^\flat \circ \eta : \pi^*\mathbb{T}_S \to \pi^*\mathbb{L}_S[n]$. 
Denote by $\sigma^\flat$ the morphism $\mathbb{T}_S \to \mathbb{L}_S[n]$ induced by the leading term of $\sigma\in \mathsf{Symp}(S,n)$. The derived pullback operator associated to $\pi^*\sigma^\flat$ then factors through the projection to the base as $p^\vee{[n]} \circ \sigma_S^\flat \circ p$, where $\sigma_S^\flat = \pi^*\sigma^\flat$. Because $p \circ i \simeq 0$, this operator vanishes on the relative tangent complex $\mathbb{T}_{X/S}$. Thus, its block matrix representation is horizontal:
\begin{equation} \label{eq:sigma_block}
    \pi^*\sigma^\flat \simeq \begin{pmatrix} 0 & 0 \\ 0 & \sigma_S^\flat \end{pmatrix}.
\end{equation}
Because $(S, \sigma)$ is an $n$-shifted symplectic stack, $\sigma^\flat$ is an equivalence \cite[Definition 1.18]{PTVV}, making $\sigma_S^\flat$ an equivalence. Consequently, the total candidate operator $\omega_X^\flat = \Omega^\flat + c \cdot \pi^*\sigma^\flat$ induced from the form $\omega_X$ in   (\ref{eq:omega_candidate}) evaluates to a block-diagonal matrix:
\begin{equation} \label{eq:omega_x_block}
    \omega_X^\flat \simeq \begin{pmatrix} \Omega_{rel}^\flat & 0 \\ 0 & C + c \cdot \sigma_S^\flat \end{pmatrix}.
\end{equation}

In an additive category, a block-diagonal morphism is an isomorphism if and only if both diagonal blocks are isomorphisms. Since the top-left block $\Omega_{rel}^\flat$ is an equivalence (established in Step 2), $\omega_X^\flat$ is an isomorphism in $\mathrm{Ho}({QCoh}(X))$ if and only if the bottom-right horizontal block $C + c \cdot \sigma_S^\flat$ is an equivalence. 

Thus, it remains to show that the bottom-right block is an equivalence, which will be discussed in the next step.

\medskip \noindent \textbf{Step 4: Generic invertibility and properness.} \
Factoring out the equivalence $\sigma_S^\flat$, the non-degeneracy condition on the horizontal block becomes:
\begin{equation} \label{eq:horizontal_block}
    C + c \cdot \sigma_S^\flat \simeq \sigma_S^\flat \circ \left( c \cdot \mathrm{id}_{\pi^*\mathbb{T}_S} + (\sigma_S^\flat)^{-1} \circ C \right).
\end{equation}
Let $M = (\sigma_S^\flat)^{-1} \circ C$. This $M$ is a global endomorphism of the perfect complex $\pi^*\mathbb{T}_S$.  To ensure non-degeneracy, we must choose $c \in \K$ such that $c \cdot \mathrm{id} + M$ is an isomorphism in $\mathrm{Ho}({QCoh}(X))$. 

We frame the argument over function fields. Since $X$ is quasi-compact and locally of finite presentation over $\K$, the underlying topological space of $t_0(X)$ is Noetherian, so it has finitely many irreducible components, with generic points $\xi_1, \dots, \xi_m$. Evaluate the endomorphism algebra at each $\xi_j$. At $\xi_j$, the algebra reduces to a finite-dimensional graded algebra over the function field $\kappa(\xi_j)$. Over a field, the endomorphism $\xi_j^* M$ possesses a well-defined characteristic polynomial in each cohomology degree, giving a finite set of roots $E_j \subset \overline{\kappa(\xi_j)}$. 

Because the base field $\K$ is infinite, we can pick a generic scalar $c \in \K$ that avoids all roots at these generic points across all irreducible components:
$$-c \notin \bigcup_{j=1}^m (E_j \cap \K).$$
With this choice, $c \cdot \mathrm{id} + M$ is an equivalence at all generic points $\xi_j$. Let $K'$ be the mapping cone of $c \cdot \mathrm{id} + M$. Since $K'$ is perfect (\cite[Proposition 1.2.3.7]{ToenHAG}), the locus where $K' \simeq 0$ formally defines a derived Zariski-open immersion (\cite[Proposition 1.2.10.1]{ToenHAG}). Since invertibility holds at all generic points $\xi_j$, the openness of invertibility yields a dense Zariski-open derived Artin substack $U \subseteq X$. On $U$, the horizontal block is an equivalence, and thus $\omega_X$ defines an absolute $n$-shifted symplectic structure globally on $U$.

Furthermore, if $X$ is proper over $\K$, then since $S$ is locally of finite presentation and admits an $n$-shifted symplectic structure (so $\mathbb{L}_S \simeq \mathbb{T}_S[-n]$ forces bounded coherence on both sides, making $\mathbb{T}_S$ a perfect complex), the pullback $\pi^*\mathbb{T}_S$ is a perfect complex globally on $X$. Therefore, the $\K$-algebra of global endomorphisms $$A = \operatorname{Ext}^0_X(\pi^*\mathbb{T}_S, \pi^*\mathbb{T}_S)$$ is a finite-dimensional $\K$-vector space. Consequently, the global endomorphism $M \in A$ satisfies a minimal polynomial $P(t) \in \K[t]$. The set $R$ of roots of $P(t)$ in $\K$ is finite. By choosing a constant $c \in \K$ such that $-c \notin R$, the element $c \cdot \mathrm{id} + M$ is invertible in the algebra $A$, which means it is automatically a global isomorphism in $\mathrm{Ho}({QCoh}(X))$. In this case, the degeneracy locus is empty ($U = X$), and the absolute $n$-shifted symplectic structure $\omega_X$ extends globally.

We then complete the proof of Theorem \ref{thm: proof_thurston in dsg}, and hence that of Theorem \ref{thm:1}.

\end{proof}

We close the section with the most obvious (basic) example: 
\begin{example}
    If $(X,\omega_X)$ and $(Y,\omega_Y)$ are both $n$-shifted symplectic $\K$-stacks, then $X\times Y$ admits a \bfem{product} $n$-shifted symplectic form $\omega_{X\times Y}=\pi^*_X \omega_X + \pi^*_Y \omega_Y$ where $\pi_X: X\times Y \to X$ and $\pi_Y:X\times Y \to Y$ are the projection maps. Then $\pi_X$ and $\pi_Y$ are \bfem{trivial} $n$-shifted symplectic fibrations, with relative structures the images of $\pi_Y^*\omega_Y$ and $\pi_X^*\omega_X$ in relative forms, and both are compatible with $\omega_{X\times Y}$. 
\end{example} 


\section{Applications of derived Thurston theorem} \label{sec:applications}
\subsection{Conormal stacks, mapping stacks and moment map quotients}
The following results apply the derived Thurston theorem to construct absolute shifted symplectic structures on the source stacks of the shifted symplectic fibrations given in Example \ref{example:conormal stack}, Example \ref{ex:mapping_stack}, and Example \ref{ex:moment map}.

\begin{proposition}[Conormal Stacks]\label{prop: conormal} 
    Let $f: X \rightarrow S$ be a morphism of derived Artin $\K$-stacks locally of finite presentation, with $N^*[n]f$ and $S$ quasi-compact, such that $S$ admits an $(n-1)$-shifted symplectic structure $\sigma$. Assume additionally that there exists an $(n-1)$-shifted closed 2-form $\Omega \in \mathcal{A}^{2,cl}(N^*[n]f, n-1)$ compatible, in the sense of Definition \ref{def:compatibility}, with the $(n-1)$-shifted symplectic fibration structure of $\pi: N^*[n] f \rightarrow S$ (Example \ref{example:conormal stack}). Then there exists an $(n-1)$-shifted symplectic structure on a dense Zariski-open derived substack of the $n$-shifted conormal $\K$-stack $N^*[n] f$ which is compatible with the fibration $\pi: N^*[n] f\longrightarrow S$.
\end{proposition}

\begin{proof}
    The canonical fiber sequence $f^*\mathbb{L}_S \rightarrow \mathbb{L}_X \rightarrow \mathbb{L}_{X/S}$ induces an $n$-shifted Lagrangian structure on the morphism $g:N^*[n] f \longrightarrow T^*[n]S$, where the target is equipped with its canonical $n$-shifted symplectic structure $d_{dR}\lambda$ ($\lambda$ being the natural Liouville 1-form). These fit into the commutative diagram:
    
    \begin{equation}
        \begin{tikzcd}
N^*[n]f \arrow[r, "g"] \arrow[rd, "{\pi}"', dashed] & T^*[n]S \arrow[d, "\pi_S"] \\
       & S                        
\end{tikzcd}
    \end{equation}  
where $\pi_S$ is an $n$-shifted Lagrangian fibration \cite[Example 1.8]{Safronov2023}. Consequently, the composition $\pi=\pi_S \circ g$ carries an $(n-1)$-shifted symplectic fibration structure \cite[Proposition 1.10]{Safronov2023}. Thus, there is a relative $(n-1)$-shifted symplectic structure $\omega$ on $\pi$, i.e., $\omega \in \mathcal{A}^{2,cl}(N^*[n]f/S, n-1)$ and its restriction $\omega_s$ to any fiber $\pi^{-1}(s)$ defines an $(n-1)$-shifted symplectic structure. Let us first recall from \cite[Theorem 2.23]{Safronov2019} the construction of $\omega_s$ from $d_{dR}\lambda$.
For each $s\in S(\K)$, consider 
\[
\begin{tikzcd}
{
{N^*[n] f \times_{T^*[n]S} (T^*[n]S)_s}} \arrow[rr, "h"] \arrow[d, "j_s", hook] &  & {(T^*[n]S)_s} \arrow[r, "\pi_S|"] \arrow[d, "i_s", hook] & \{s\} \arrow[d, "i", hook] \\
{N^*[n] f} \arrow[rr, "g"] \arrow[rrr, "\pi=\pi_S \circ g", bend right]            &  & {T^*[n]S} \arrow[r, "\pi_S"]                             & S                          
\end{tikzcd}
\]  
where $i_s$ and $j_s$ denote the inclusions of the fibers $$\pi_S^{-1}(s)=(T^*[n]S)_s \quad \text{ and } \quad \pi^{-1}(s)=N^*[n] f \times_{T^*[n]S}(T^*[n]S)_s,$$ respectively. Here $N^*[n] f \times_{T^*[n]S} (T^*[n]S)_s$ is the fiber product of the left square.

As shown in the proof of \cite[Theorem 2.23]{Safronov2019}, the commutativity of the left rectangle and the Lagrangian structures on $g$ and $i_s$ yield a composition of homotopies $$0 \sim j_s^*g^*d_{dR}\lambda \sim h^*i_s^*d_{dR}\lambda \sim 0$$ in $\mathcal{A}^{2,cl}(\pi^{-1}(s), n)$. This loop at the zero form defines a class in $\pi_1(\mathcal{A}^{2,cl}(\pi^{-1}(s), n), 0)$. By the Dold-Kan correspondence, taking the loop space drops the cohomological degree by $1$, so this canonically evaluates to an element $\omega_s \in \pi_0(\mathcal{A}^{2,cl}(\pi^{-1}(s), n-1)) \simeq \mathcal{A}^{2,cl}(\pi^{-1}(s), n-1)$. Moreover, by \cite[Theorem 2.9]{PTVV}, $\omega_s$ is non-degenerate, so it defines an $(n-1)$-shifted symplectic structure on the fiber $\pi^{-1}(s)$ for each $s\in S(\K)$. In the proof of \cite[Proposition 1.10]{Safronov2023}, the relative structure $\omega$ obtained above comes from the Lagrangian intersection, relative to $S$, of $g \times \mathrm{id}_S$ and $\mathrm{id} \times \pi_S: T^*[n]S \rightarrow T^*[n]S \times S$, which we form with $g \times \mathrm{id}_S$ first. Its base change along $s: \spec\, \K \rightarrow S$ is the Lagrangian intersection of $g$ and $i_s$, formed with $g$ first as in the loop above, where $i_s$ carries the Lagrangian structure of \cite[Corollary 1.7]{Safronov2023}. So these fiberwise structures are, up to homotopy, the restrictions $j_s^*\omega$ of $\omega$. 

We now apply the derived Thurston theorem (Theorem \ref{thm: proof_thurston in dsg}). By our assumption, the closed 2-form $\Omega \in \mathcal{A}^{2,cl}(N^*[n]f, n-1)$ is compatible with $\omega$, that is, $j_s^*\Omega \sim \omega_s$ in $\mathcal{A}^{2,cl}(\pi^{-1}(s), n-1)$ for each $s \in S(\K)$. By Theorem \ref{thm: proof_thurston in dsg}, with a suitable choice of a constant $c \in \K$, the 2-form $\Omega_{N^*[n]f} := \Omega + c \cdot \pi^*\sigma$ satisfies $\Omega_{N^*[n]f} \in \mathsf{Symp}(U, n-1)$ for a dense open substack $U \subseteq N^*[n]f$, and $j_s^*\Omega_{N^*[n]f} \sim \omega_s$ for each $s \in S(\K)$. If $N^*[n]f$ is proper, this extends globally.

\end{proof}

Note that the above result can also be extended verbatim to the map $N_{\beta}^*[n] f \to T_{\beta}^*[n]S$ of twisted $\K$-stacks for each $n$. Namely, via minor modifications of the above proof, one can readily establish:

\begin{corollary}[Twisted Conormal Stacks] \label{cor: conormal}
    Let $f: X \rightarrow S$ be a morphism of derived Artin $\K$-stacks locally of finite presentation, let $\beta$ be a closed $1$-form of degree $n+1$ on $S$ with a null-homotopy of $f^*\beta$ \cite[Definition 1.22]{Safronov2023}, and suppose that $N^*_\beta[n]f$ and $S$ are quasi-compact and that $S$ admits an $(n-1)$-shifted symplectic structure $\sigma$. Suppose additionally that there exists an $(n-1)$-shifted closed 2-form $\Omega \in \mathcal{A}^{2,cl}(N^*_\beta[n]f, n-1)$ compatible, in the sense of Definition \ref{def:compatibility}, with the $(n-1)$-shifted symplectic fibration structure of $\pi: N^*_\beta[n] f \rightarrow S$. 
    
    Then there exists an $(n-1)$-shifted symplectic structure on a dense Zariski-open derived substack of the $n$-shifted $\beta$-twisted conormal $\K$-stack $N^*_\beta[n] f$ which is compatible with the fibration $\pi: N^*_\beta[n] f\longrightarrow S$. 
\end{corollary}
\medskip
 
Now, let us consider the situation in Example \ref{ex:moment map} and apply Theorem \ref{thm: proof_thurston in dsg} for $2$-shifted symplectic fibration structure on $X/G \longrightarrow (BG, \omega_{BG})$, along with an additional condition.
\begin{proposition}[Moment Map Quotients] \label{prop: moment quotient}
    Let $(X, \omega_X)$ be a {$2$-shifted} symplectic derived Artin stack equipped with a smooth action of $G$, and let $\mu: X \rightarrow {\mathfrak{g}^*[2]}$ be a \emph{{$2$-shifted} moment map}. Assume that the quotient stack $X/G$ and $BG$ are quasi-compact, and that there exists a $2$-shifted closed 2-form $\Omega$ on the quotient stack $X/G$ such that its pullback under the map $p: X \rightarrow X/G$ is equivalent to $\omega_X$. Then there exists a {$2$-shifted} symplectic structure on a dense Zariski-open derived substack of $X/G$.
\end{proposition}

\begin{proof}
The composition $\Pi: X/G \longrightarrow \mathfrak{g}^*[2]/G \simeq T^*[3]BG\longrightarrow(BG, \omega_{BG})$, with the form $\omega_{BG} \in \mathsf{Symp}(BG,2),$ carries a $2$-shifted symplectic fibration structure arising from the following pullback diagram:
\[
\begin{tikzcd}
{{X/G \times_{T^*[3]BG} \mathfrak{g}^*[2]}} \arrow[rr] \arrow[d] &  & {\mathfrak{g}^*[2]} \arrow[r] \arrow[d,"\text{3-Lag}"] & * \arrow[d] \\
{X/G} \arrow[rr, "\text{3-Lag}"]      &  & {{T^*[3]BG}} \arrow[r, "\pi_{BG}"]                             & (BG, \omega_{BG}).
\end{tikzcd}
\]Since $\mu$ is a moment map, there is an equivalence of $2$-shifted symplectic stacks \cite[Definition 4.3.2]{GrataloupPHD}\[ X/G \times_{T^*[3]BG} \mathfrak{g}^*[2] \simeq X,\]which implies that the symplectic structure on each fiber is equivalent to the original structure $\omega_X$ on $X$. This structure is the restriction $p^*\omega$ of the relative structure $\omega$ of Example \ref{ex:moment map} to the fiber $X$ over the base point of $BG$, by base change along $\spec\, \K \rightarrow BG$ of the relative Lagrangian intersection of \cite[Proposition 1.10]{Safronov2023}, formed with $X/G \rightarrow T^*[3]BG$ first as in \cite[Definition 4.3.2]{GrataloupPHD}. Since $\K$ is algebraically closed, every $\K$-point of $BG$ is isomorphic to the base point. So the assumption $p^*\Omega \sim \omega_X$ says that $\Omega$ is compatible with $\omega$ in the sense of Definition \ref{def:compatibility}, which is the hypothesis of Theorem \ref{thm: proof_thurston in dsg}. 
The desired symplectic form on a dense open substack of $X/G$ is then constructed using the form $\Omega$ and the pullback of $\omega_{BG}$ under the composition $\Pi$.
\end{proof}

As a further application, the derived Thurston theorem (Theorem \ref{thm: proof_thurston in dsg}) applies to the map $\mathsf{Map}(X,L) \rightarrow \mathsf{Map}(X,S)$ of Example \ref{ex:mapping_stack}, with the symplectic fibration structure transgressed from that of $p \circ f$.

\begin{proposition}[Mapping Stacks] \label{prop: mapping_stack_thurston}
    Let $Y$ be an $n$-shifted symplectic derived Artin $\K$-stack and $X$ be a smooth projective Calabi-Yau $m$-fold. Suppose $f: L \rightarrow Y$ is an $n$-shifted Lagrangian morphism and $p: Y \rightarrow S$ is an $n$-shifted Lagrangian fibration such that the base $S$ admits an $(n-1)$-shifted symplectic structure $\sigma_S$. 
    
    
    Assume additionally that the mapping stacks $\mathsf{Map}(X,L)$ and $\mathsf{Map}(X,S)$ are quasi-compact (e.g., by restricting to appropriate open substacks or connected components), and that there is an absolute $(n-1)$-shifted closed 2-form $\Omega_L \in \mathcal{A}^{2,cl}(L, n-1)$ whose image in $\mathcal{A}^{2,cl}(L/S, n-1)$ is homotopic to the relative $(n-1)$-shifted symplectic structure $\omega_L$ of the induced symplectic fibration $p \circ f: L \rightarrow S$ (from Proposition \ref{prop: way of inducing symp fib structure}).

    Let $[X]$ be the $\mathcal{O}$-orientation of dimension $m$ of $X$ given by a trivialization of its canonical bundle \cite[Section 2.1]{PTVV}. Then the transgression $\omega_\Pi := \int_{[X]} \mathrm{ev}^*\omega_L$ of $\omega_L$ along the evaluation map $\mathrm{ev}: X \times \mathsf{Map}(X,L) \rightarrow L$ is an $(n-m-1)$-shifted symplectic fibration structure on $\Pi: \mathsf{Map}(X,L) \longrightarrow \mathsf{Map}(X,S)$, and there exists an $(n-m-1)$-shifted symplectic structure on a dense Zariski-open derived substack of the mapping stack $\mathsf{Map}(X,L)$ which is compatible with $\omega_\Pi$.
\end{proposition}
\begin{proof}
        Composition with $f$ and $p$ gives the diagram of Example \ref{ex:mapping_stack}
    \begin{equation*}
        \begin{tikzcd}
        \mathsf{Map}(X,L) \arrow[r, "f \circ -"] \arrow[rd, "\Pi"', dashed] & \mathsf{Map}(X,Y) \arrow[d, "p \circ -"] \\
           & \mathsf{Map}(X,S).
        \end{tikzcd}
    \end{equation*}
    
    By \cite[Corollary 3.3]{Toen2014}, the mapping stacks in this diagram are derived Artin $\K$-stacks locally of finite presentation, since $X$ is proper. The evaluation map $\mathrm{ev}: X \times \mathsf{Map}(X,L) \rightarrow L$ lies over the evaluation map $X \times \mathsf{Map}(X,S) \rightarrow S$, so it pulls back forms relative to $S$ to forms relative to $X \times \mathsf{Map}(X,S)$. The map $X \times \mathsf{Map}(X,L) \rightarrow X \times \mathsf{Map}(X,S)$ is the base change of $\Pi$ along the projection $X \times \mathsf{Map}(X,S) \rightarrow \mathsf{Map}(X,S)$. For the relative de Rham algebras of \cite[Definitions 2.3.1 and 2.4.2]{CPTVV}, this and the $\mathcal{O}$-compactness of $X$ \cite[Definition 2.1]{PTVV} give an equivalence of graded mixed complexes $$DR(X \times \mathsf{Map}(X,L)/X \times \mathsf{Map}(X,S)) \simeq DR(\mathsf{Map}(X,L)/\mathsf{Map}(X,S)) \otimes_{\K} C(X,\mathcal{O}_X),$$ where $C(X,-)$ denotes derived global sections. Composing it with $\mathrm{id} \otimes [X]$ gives the integration along $[X]$ relative to $\mathsf{Map}(X,S)$, as in \cite[Definition 2.3]{PTVV}. On an affine $F$, the map $\kappa_{F,X}$ used in \cite[Definition 2.3]{PTVV} is the K\"unneth equivalence, then the projection of $DR(\mathcal{O}_X)$ onto its weight $0$ part $\mathcal{O}_X$, then the identification $C(X, DR(F) \otimes \mathcal{O}_X) \simeq DR(F) \otimes_{\K} C(X,\mathcal{O}_X)$. The projection from absolute forms to forms relative to $X \times \mathsf{Map}(X,S)$ also projects $DR(\mathcal{O}_X)$ onto $\mathcal{O}_X$, since $X$ is a factor of the base. So the relative integration commutes with $d_{dR}$ and with the projections from absolute to relative forms. It gives the relative closed 2-form $$\omega_\Pi = \int_{[X]} \mathrm{ev}^*\omega_L \in \mathcal{A}^{2,cl}(\mathsf{Map}(X,L)/\mathsf{Map}(X,S), n-m-1).$$ This form is non-degenerate by the argument of the proof of \cite[Theorem 2.5]{PTVV}. Namely, at an $A$-point of $\mathsf{Map}(X,L)$ given by $g: X \times \spec\, A \rightarrow L$, the pullback of the relative tangent complex of $\Pi$ is $C(X \times \spec\, A, g^*\mathbb{T}_{L/S})$, and the pairing induced by $\omega_\Pi$ is the one induced by $\omega_L$ and $[X]$. This pairing is non-degenerate since $\omega_L$ is non-degenerate and $[X]$ is an $\mathcal{O}$-orientation \cite[Definition 2.4]{PTVV}. So $\omega_\Pi$ is an $(n-m-1)$-shifted symplectic fibration structure on $\Pi$. We do not compare $\omega_\Pi$ with the fibration structure on $\Pi$ of Example \ref{ex:mapping_stack}, which comes from the Lagrangian structure on $f \circ -$ and the Lagrangian fibration structure on $p \circ -$.

    Set $$\Omega := \int_{[X]} \mathrm{ev}^* \Omega_L \in \mathcal{A}^{2,cl}(\mathsf{Map}(X,L), n-m-1).$$ Since the integration commutes with the projections to relative forms, the image of $\Omega$ in relative forms over $\mathsf{Map}(X,S)$ is the integral of the image of $\Omega_L$ in $\mathcal{A}^{2,cl}(L/S, n-1)$. By assumption the image of $\Omega_L$ is homotopic to $\omega_L$, so the image of $\Omega$ is homotopic to $\omega_\Pi$. As in the proof of Theorem \ref{thm:2}, the restriction of forms to the fiber of $\Pi$ over a point $\varphi \in \mathsf{Map}(X,S)(\K)$ factors through relative forms over $\mathsf{Map}(X,S)$. Hence $\Omega$ is compatible with $\omega_\Pi$ in the sense of Definition \ref{def:compatibility}. The same argument applied to $p \circ f$ shows that $\Omega_L$ is compatible with $\omega_L$.

    On the other hand, the base stack $S$ is equipped with the absolute $(n-1)$-shifted symplectic structure $\sigma_S$. By the AKSZ construction \cite[Theorem 2.5]{PTVV}, the mapping stack $\mathsf{Map}(X, S)$ out of the Calabi-Yau $m$-fold $X$, with the $\mathcal{O}$-orientation $[X]$, naturally inherits an $(n-1)-m = (n-m-1)$-shifted symplectic structure $\sigma_{\mathsf{Map}}$. 
    
    In total, we end up with the situation where the map$$\Pi: \mathsf{Map}(X,L) \longrightarrow (\mathsf{Map}(X,S), \sigma_{\mathsf{Map}})$$ is an $(n-m-1)$-shifted symplectic fibration, with structure $\omega_\Pi$, over an $(n-m-1)$-shifted symplectic base, along with a global absolute $(n-m-1)$-shifted closed 2-form $\Omega$ on $\mathsf{Map}(X,L)$ compatible with $\omega_\Pi$. This matches the setup of the Derived Thurston Theorem (Theorem \ref{thm: proof_thurston in dsg}), since the mapping stacks are derived Artin $\K$-stacks locally of finite presentation and quasi-compact. 
    Theorem \ref{thm: proof_thurston in dsg} then applies, and we conclude that for a generic scalar $c \in \K$, the 2-form $$\Omega_{\mathsf{Map}} := \Omega + c \cdot \Pi^*\sigma_{\mathsf{Map}}$$ defines an absolute $(n-m-1)$-shifted symplectic structure on a dense Zariski-open derived substack of $\mathsf{Map}(X, L)$. By Theorem \ref{thm: proof_thurston in dsg}(ii), this structure is compatible with $\omega_\Pi$.
\end{proof}

\subsection{Affine model construction}\label{sec:affine symplectic fibration}
In this section, we present an affine local model for an $n$-shifted symplectic fibration structure on the map $\spec\beta: \spec A \longrightarrow (\spec B, \sigma)$ of derived affine schemes, induced from a submersion $\beta:B\rightarrow A$ of standard form cdgas. Furthermore, by choosing a natural closed 2-form on $\spec A$ that satisfies the hypothesis of Theorem \ref{thm: proof_thurston in dsg}, we explicitly describe a shifted symplectic structure on the source $\spec A$.

\paragraph{Step-1: Choosing $(\spec B, \sigma)$ in Darboux form.} Assume that $n<0$ is odd; say, $n=-2\ell-1$ for $\ell \in \N$. 
Suppose $(B, \sigma)$ is in $n$-shifted symplectic Darboux form. Briefly, we assume that $B$ is a (minimal) standard form cdga given as a free algebra over a smooth $\K$-algebra $B(0)$ with coordinates $x_j^0$, generated by the variables $x^{-i}_j, y^{n+i}_j \in B$, where
    \begin{align} 
    & x_1^{-i}, \dots, x_{m_i}^{-i}& &\text{ in degree } -i \ \ \ \text{ for } i= 1,  \dots, \ell, \label{var set1} \\
    & y_1^{n+i}, \dots, y_{m_i}^{n+i}& & \text{ in degree } n+i \ \text{ for } i=0,\dots, \ell, \label{var set2}
    \end{align} where $m_1,\dots, m_{\ell}$ are non-negative integers. The \emph{internal differential} $d_B$ on $B$ is determined by the equations: 
\begin{align} \label{defn_internal d}
    d_B|_{B(0)}=0; \quad d_Bx_j^{-i} =  \frac{\partial H}{\partial y_j^{n+i}} \ \text{ and }  \ d_By_j^{n+i} = \pm \frac{\partial H}{\partial x_j^{-i}} \text{ for all } i,j,
\end{align}    where $H\in B^{n+1}$ is the \emph{Hamiltonian}. 
Furthermore, $\Omega^1_{B}$ is the free $B$-module of finite rank with basis $ \{d_{dR}x_j^{-i}, d_{dR}y_j^{n+i} : \forall i,j \text{ including } i=0\},$ 
and $\sigma:=(\sigma^0, 0,  \dots)$, with
     $  \sigma^0= \sum_{i,j}d_{dR}x_j^{-i} d_{dR}y_j^{n+i},$ is an $n$-shifted symplectic form on $\spec B$ in Darboux form. 

\paragraph{Step-2: Setting the source $\spec A$.} Let $A^0:=A(0)$ be a smooth algebra of dimension $m_0+n_0$, and let $\beta^0:B^0\rightarrow A^0$ be a smooth morphism. Assume there exist elements $u^0_1, \dots, u^0_{n_0}$ in $A^0$ such that $\{d_{dR}\tx_{1}^0, \dots, d_{dR}\tx^0_{m_0}, d_{dR}u^0_1, \dots, d_{dR}u^0_{n_0} \}$ is a basis over $A^0$ for $\Omega_{A^0}^1$, with $\tx^0_j=\beta^0(x^0_j)$. 

Letting $s=\ell$, we define the \emph{commutative graded algebra} $A$ to be the free graded algebra over $A(0)$ generated by the variables:  
\begin{align} 
& \tx_1^{-i}, \dots, \tx_{m_i}^{-i} & & \text{ in degree } (-i) \text{ for } i= 1, \dots, \ell, \nonumber \\
& \ty_1^{n+i}, \dots, \ty_{m_i}^{n+i} & & \text{ in degree } (n+i) \text{ for } i= 0, \dots, \ell, \nonumber \\
& u_1^{-i}, \dots, u_{n_i}^{-i} & & \text{ in degree } (-i) \text{ for } i=1, \dots, s, \nonumber \\
& v_1^{n+i}, \dots, v_{n_i}^{n+i} & & \text{ in degree } (n+i) \text{ for } i=0,\dots, s.
\end{align}

Next, choose an element $G \in A^{n+1}$ (the \emph{superpotential}) depending only on the variables $u$ and $v$, satisfying the classical master equation. We define the \emph{differential $d_A$ on $A$} via:
\begin{align*}
    &d_A|_{A^0}=0, \quad d_A\tx_j^{-i}= \beta(d_B x_j^{-i}), \quad d_A\ty_j^{n+i}= \beta(d_B y_j^{n+i}), \\
    & d_Au_j^{-i}= \partial G/ \partial v_j^{n+i}, \quad d_Av_j^{n+i}= \pm \partial G/ \partial u_j^{-i}.
\end{align*} 
Here we extended $\beta^0$ to a graded algebra morphism $\beta:B\rightarrow A$ by setting $\beta(x_j^{-i}) = \tx_j^{-i}$ and $\beta(y_j^{n+i}) = \ty_j^{n+i}$.

\paragraph{Step-3: Setting the map from $\spec A$ to $\spec B$.} 
The commutation relation $d_A \circ \beta = \beta \circ d_B$ holds consistently on all generators by design, confirming that $\beta:B\rightarrow A$ is a valid morphism of \emph{cdgas}. Furthermore, $\beta$ acts as a submersion, which yields the desired morphism $$\spec\beta: \spec A \longrightarrow (\spec B, \sigma)$$ of derived affine schemes. 
 
\paragraph{Step-4: Defining a relative $n$-shifted symplectic structure on $\spec \beta.$} Note that in the relative setting, $\dR \tx_j^{-i}=0$ and $\dR \ty_j^{n+i}=0$ in $\Omega^1_{A/B}$ because they belong to the image of $\beta.$ Let $\gamma $ be a relative $n$-shifted 2-form on $\spec \beta$ defined by \[ \gamma=\sum_{i,j} \dR u_j^{-i} \dR v_j^{n+i} \in \displaystyle \bigwedge\nolimits^2 \Omega^1_{A/B}[n]. \]  
Clearly, $\dR \gamma=0$. Furthermore, we compute:
\begin{align*}
d_A\gamma &= \sum_{i,j} \left((d_A\circ\dR u_j^{-i}) \dR v_j^{n+i} \pm (d_A\circ \dR v_j^{n+i}) \dR u_j^{-i}\right) \\
&= \dR (\dR G) = 0. 
\end{align*}

To verify non-degeneracy globally, we use the properties of \bfem{minimal points} on standard form cdgas. By base change for cotangent complexes \cite[Lemma 1.4.1.16]{ToenHAG}, pulling back the relative cotangent complex to a \emph{minimal point} $p$ (a point where the internal differentials vanish, $d\vert{}_p = 0$) yields the cotangent complex at that point. Because the differentials vanish at $p$, the derived pairing reduces to classical block matrix linear algebra over graded $\K$-vector spaces. At the minimal point $p$, the induced map
$$\gamma^{\flat}\vert_p : \mathbb{T}_{A/B}\vert_p \rightarrow \Omega^1_{A/B}[n]\vert_p$$
sends $\langle \partial/\partial v_{j}^{n+i}\rangle_p \longmapsto \langle \dR u_{j}^{-i}\rangle_p$ and $\langle \partial/\partial u_{j}^{-i}\rangle_p \longmapsto \langle \dR v_{j}^{n+i}\rangle_p$ isomorphically, establishing that $\gamma^\flat$ is a quasi-isomorphism at $p$.

Having verified non-degeneracy at the minimal point $p$, we invoke the "Vanishing Locus" property. Let $K = \mathrm{Cone}(\gamma^\flat)$; by \cite[Proposition 1.2.3.7]{ToenHAG}, $K$ is perfect. By \cite[Proposition 1.2.10.1]{ToenHAG}, the locus where $K \simeq 0$ formally defines a Zariski-open derived substack $U \hookrightarrow \spec A$. Since non-degeneracy holds at $p$, the point $p$ is contained in $U$. In fact $U = \spec A$. The module $\Omega^1_{A/B}$ is the free graded $A$-module on the $\dR u_j^{-i}$ and $\dR v_j^{n+i}$. The map $\gamma^\flat$ sends the dual basis of $\mathbb{T}_{A/B}$ to these generators up to sign, because $\gamma$ has constant coefficients. So $\gamma^\flat$ is an isomorphism of graded $A$-modules. It commutes with the differentials because $d_A\gamma = 0$, so it is an isomorphism of complexes, and in particular a quasi-isomorphism at every point. We conclude that $\gamma$ defines a relative $n$-shifted symplectic structure on $\spec \beta$.

\paragraph{Step-5: Choosing a suitable $\Omega \in \mathcal{A}^{2,cl}(\spec A, n).$} Recall the null-homotopic sequence of complexes:
\[DR(B) \xrightarrow[]{(\spec \beta)^*} DR(A) \xrightarrow[]{\cdot/B} DR(A/B).\] We wish to define an $n$-shifted closed 2-form $\Omega:=(\Omega, 0, 0, \dots)$ on $\spec A$ whose image under the map $\cdot /B$ is homotopic to $\gamma.$ Using the variables above, we set:
\[ \Omega=\sum_{i,j} \dR u_j^{-i} \dR v_j^{n+i} + \sum_{i,j} \dR \tx_j^{-i} \dR \ty_j^{n+i}. \]
Notice that $\Omega-\gamma=\sum_{i,j} \dR \tx_j^{-i} \dR \ty_j^{n+i}$, and we compute:
\begin{align*}
    \Omega-\gamma &= \beta_* \bigg(\sum_{i,j} \dR x^{-i}_j \dR y^{n+i}_j\bigg) = (\spec \beta)^* (\sigma).
\end{align*}
Because $\mathbb{T}_A$ splits as $(\spec \beta)^* (\mathbb{T}_B) \oplus\mathbb{T}_{A/B}$ over $\spec A$ (and similarly for $\Omega^1_A$), it follows by construction that $d_A \Omega=0$. Thus, $\Omega$ provides the desired closed form.

\paragraph{Step-6: Constructing an $n$-shifted symplectic structure on $\spec A.$} Using the element $\Omega$ introduced above, we simply set $\Theta:= \Omega$ and define \[ \omega_A:= \Theta + c (\spec \beta)^* (\sigma) \quad \text{ for a scalar } c\in \K.\] 

To verify that $\omega_A$ is non-degenerate, observe that by construction there are no cross-terms in $\Omega$ between the relative vertical variables (the $u$ and $v$ generators) and the horizontal variables (the $\tx$ and $\ty$ generators). Consequently, the induced morphism $\omega_A^\flat: \mathbb{T}_A \rightarrow \mathbb{L}_A[n]$ splits explicitly as a block-diagonal matrix over the decomposition $$\mathbb{T}_A \simeq (\spec \beta)^* (\mathbb{T}_B) \oplus \mathbb{T}_{A/B}.$$ Specifically, its block-diagonal components are exactly $(1+c)(\spec \beta)^*(\sigma)^\flat$ and $\gamma^\flat$. Since $\gamma$ is a relative $n$-shifted symplectic structure (making $\gamma^\flat$ an equivalence vertically) and $\sigma$ is an $n$-shifted symplectic structure on the base (making $\sigma^\flat$ an equivalence horizontally), their induced block morphisms are independently global equivalences. Therefore, the combined block matrix is globally invertible on $\spec A$ for any constant $c \in \K$ such that $1+c \neq 0$. So for every $c \neq -1$ the morphism $\omega_A^\flat$ is an equivalence, $\omega_A$ is non-degenerate, and it is an absolute $n$-shifted 
symplectic structure on $\spec A$.

\qed
\section{Concluding remarks and future directions}

This paper extends the concept of symplectic fibration to the setting of derived symplectic geometry with the help of relative differential calculus. In this context, the $n$-shifted relative symplectic structures serve as an appropriate substitute for classical symplectic fibrations. 

Regarding such structures, we first provide Proposition \ref{prop: way of inducing symp fib structure}, which introduces a way of inducing symplectic fibration structures using the composition of Lagrangians and Lagrangian fibrations. As applications of that gadget, we present several examples (mostly gathered from literature), consisting of conormal stacks, mapping stacks, and moment map quotients, see Example \ref{example:conormal stack}, Example \ref{ex:mapping_stack}, and Example \ref{ex:moment map}.

We established two key theorems in this paper, Theorem \ref{thm:2} and Theorem \ref{thm: proof_thurston in dsg}. The former serves as the derived analogue of the \textit{induced compatible symplectic fibration lemma} \cite[Lemma 6.1.3]{Mcduff}. The latter, called the \textit{derived Thurston theorem}, provides a method for constructing a compatible (absolute) shifted symplectic structure on a dense Zariski-open derived substack of the source $X$ of a symplectic fibration \( \pi: X \rightarrow (S,\omega_S) \) with a shifted symplectic target, and on all of $X$ when $X$ is proper. 

As an application, we combine, under certain conditions, Theorem \ref{thm: proof_thurston in dsg} with the aforementioned examples of symplectic fibrations to produce relative-to-absolute-type constructions (Corollary \ref{cor: app to Thm1}). Finally, Section \ref{sec:affine symplectic fibration} examines affine shifted symplectic fibrations and constructs a natural local model in accordance with  Theorem \ref{thm: proof_thurston in dsg}.

 Inspired by a prior work of Arıkan \cite{mfa}, our relative-to-absolute-construction formalism can also be naturally applied to the case of \textit{contact symplectic fibrations}, in which one considers symplectic fibrations with a contact target. 
 It has been shown in \cite{mfa} that Thurston's theorem can be adapted to the framework of such fibrations, allowing the construction of a \textit{compatible} contact structure on the source space. In this regard, we aim to extend the main results of that paper to the setting of \textit{derived contact geometry} \cite{kib1} in subsequent work and investigate further consequences with applications.

\bibliographystyle{alpha}
{\small \bibliography{refss}}

\end{document}